\documentclass[10pt,fleqn,a4paper]{scrartcl}

\usepackage[english]{babel}
\usepackage[T1]{fontenc}
\usepackage[utf8]{inputenc}

\newcommand{\captionstringtheorem}{Theorem}
\newcommand{\captionstringlemma}{Lemma}
\newcommand{\captionstringcorollary}{Corollary}
\newcommand{\captionstringlemmadef}{Lemma and definition}
\newcommand{\captionstringtheoremdef}{Theorem and definition}
\newcommand{\captionstringdefinition}{Definition}
\newcommand{\captionstringproposition}{Proposition}
\newcommand{\captionstringexample}{Example}
\newcommand{\captionstringconjecture}{Conjecture}
\newcommand{\captionstringconvention}{Convention}

\usepackage{amsmath,  
            amssymb,  
            amsthm,   
            amsfonts  
            }

\usepackage{mathtools}

\usepackage{braket}   

\usepackage{extpfeil} 

\newcommand{\IZ}{\mathbb{Z}}

\newcommand{\isomorphic}{\cong}

\DeclareMathOperator{\id}{id}
\DeclareMathOperator{\Hom}{Hom}
\DeclareMathOperator{\End}{End}

\DeclareMathOperator{\Ind}{Ind}
\DeclareMathOperator{\Res}{Res}

\DeclareMathOperator{\ord}{ord}

\newcounter{theoremnumber}
\numberwithin{theoremnumber}{section}

\newtheoremstyle{dotless} 
			{\bigskipamount}    
			{\bigskipamount}    
			{\nopagebreak}      
			{}                  
			{\bfseries}         
			{:}                 
			{\newline}          
			{}                  
\newtheoremstyle{dotless2} 
			{\bigskipamount}    
			{\bigskipamount}    
			{}                  
			{}                  
			{\bfseries}         
			{:}                 
			{0.5em}             
			{}                  
\swapnumbers

\theoremstyle{dotless2}
\newtheorem{remark}[theoremnumber]{Remark}

\theoremstyle{dotless}
\newtheorem{theorem}[theoremnumber]{\captionstringtheorem}

\newtheorem{lemma}[theoremnumber]{\captionstringlemma}
\newtheorem{lemmadef}[theoremnumber]{\captionstringlemmadef}
\newtheorem{corollary}[theoremnumber]{\captionstringcorollary}
\newtheorem{proposition}[theoremnumber]{\captionstringproposition}
\newtheorem{definition}[theoremnumber]{\captionstringdefinition}
\newtheorem{example}[theoremnumber]{\captionstringexample}

\newtheorem*{convention}{\captionstringconvention}

\allowdisplaybreaks[1]

\usepackage{tikz}

\tikzset
{
	desc/.style=
	{
		fill=white,inner sep=2pt,font=\scriptsize
	},
	Vertex/.style =
	{
		inner sep = 1pt,
		outer sep = 2pt,
		minimum size=15pt,
		circle,
		draw=black!70,
		thick,
		font=\scriptsize
	},
	EdgeT/.style =
	{
		ultra thick,
		draw=black,
		font=\scriptsize
	},
	EdgeI/.style =
	{
		->,
		draw=black!70,
		font=\scriptsize
	}
}

\theoremstyle{dotless}
\newtheorem{algorithm}[theoremnumber]{Algorithm}

\usepackage[normalem]{ulem} 

\usepackage{enumitem}

\usepackage{csquotes}

\usepackage[numbers]{natbib}
\usepackage[
	pdfpagelabels=true,
	plainpages=false
	]{hyperref}
	
\hypersetup{
colorlinks=true,
linkcolor=red,
urlcolor=blue,
citecolor=blue,
linktocpage=true, 
pdfpagelayout={OneColumn},
pdfstartview= 
}

\author{Johannes Hahn}
\title{On canonical bases and induction of $W$-graphs}

\hypersetup{
pdfinfo=
	{  
		Title={On canonical bases and induction of W-graphs},
		Author={Johannes Hahn},
		Keywords={Canonical basis, W-graphs, induction, Kazhdan-Lusztig polynomial},
		Subject={algebra, representation theory}
	}
}

\begin{document}
\maketitle

\begin{abstract}
A canonical basis in the sense of Lusztig is a basis of a free module over a ring of Laurent polynomials that is invariant under a certain semilinear involution and is obtained from a fixed \enquote{standard basis} through a triangular base change matrix with polynomial entries whose constant terms equal the identity matrix.

Among the better known examples of canonical bases are the Kazhdan-Lusztig basis of Iwahori-Hecke algebras (see \cite{KL}), Lusztig's canonical basis of quantum groups (see \cite{lusztig1990canonical}) and the Howlett-Yin basis of induced $W$-graph modules (see \cite{howlett2003inducingI} and \cite{howlett2004inducingII}).

This paper has two major theoretical goals: First to show that having bases is superfluous in the sense that canonicalisation can be generalized to non-free modules. This construction is functorial in the appropriate sense. The second goal is to show that Howlett-Yin induction of $W$-graphs is well-behaved a functor between module categories of $W$-graph-algebras that satisfies various properties one hopes for when a functor is  called \enquote{induction}, for example transitivity and a Mackey theorem.
\end{abstract}

\section{Introduction}

The ring $\IZ[v^{\pm 1}]$ of Laurent polynomials has an involutive automorphism $\overline{\phantom{m}}$ defined by $\overline{v}:=v^{-1}$.

If $M$ is a free $\IZ[v^{\pm 1}]$-module equipped with an $\overline{\phantom{m}}$-semilinear involution $\iota$ and a "standard" basis $(t_x)_{x\in X}$ then a \emph{canonical basis w.r.t. $(t_x)$ and $\iota$} in the sense of Lusztig is a basis $(c_x)$ of $M$ such that $\iota(c_x)=c_x$ and $c_x \in t_x + \sum_{y\in X} v\IZ[v] t_y$ hold.

Kazhdan and Lusztig showed in \cite{KL} that the Iwahori-Hecke algebra of any Coxeter group $(W,S)$ has a canonical basis w.r.t. the standard basis $(T_w)_{w\in W}$ and the involution $\iota(T_w) = T_{w^{-1}}^{-1}$ which is now known simply as the Kazhdan-Lusztig basis. (Note that Lusztig used a slightly different construction in \cite{lusztig2003hecke} which essentially replaces $v\IZ[v]$ with $v^{-1}\IZ[v^{-1}]$ though this does not change results significantly) The special features of the action of the Hecke algebra on itself w.r.t. this basis are captured in the definition of $W$-graphs in the same paper.

In \cite{howlett2003inducingI} Howlett and Yin showed that given any parabolic subgroup $W_J\leq W$ and a $W_J$-graph $(\mathfrak{C},I,m)$ representing the $H_J$-module $V$, then the induced module $\Ind_{H_J}^H(V) := H\otimes_{H_J} V$ is also represented by a $W$-graph. They constructed the $W$-graph explicitly in terms of a canonical basis of $\Ind_{H_J}^H(V)$ and developed their ideas of inducing $W$-graphs further in \cite{howlett2004inducingII}.

In \cite{Gyoja} Gyoja proved that given any \emph{finite} Coxeter group $(W,S)$ all complex representations of the Hecke algebra can in fact be realized by a $W$-graph. His proof was not constructive but introduced the $W$-graph algebra as an auxiliary object which I investigated further in my thesis \cite{hahn2013diss} and in my previous paper \cite{hahn2014wgraphs1}. The fundamental property of the $W$-graph algebra $\Omega$ is that the Hecke-algebra is canonically embedded into $\IZ[v^{\pm 1}]\Omega$ in such a way that a representation of $H$ given by a $W$-graph canonically extends to a representation of $\Omega$. And conversely given an $\Omega$-module with a sufficiently nice basis, a $W$-graph can be obtained that realises the associated representation of $H$ (see \ref{w_graph_algebra:H_vs_Omega_modules} for details). In this sense $W$-graphs can (and I'm advocating that they should) be understood not as combinatorial objects encoding certain matrices but as modules of an algebra.

\medbreak
This paper is organised as follows: Section \ref{section:canonicalisation} is about modules over (generalized) Laurent polynomial rings equipped with an $\overline{\phantom{m}}$-semilinear involution. It defines canonical modules and canonicalisations of modules. The main theorem in this section is theorem \ref{canonicalisation_from_shadows} which proves a sufficient condition to recognize canonical modules and also shows that under the conditions present in the context of Hecke algebras (though no reference to Hecke algebras is made in this section) the canonicalisation is unique and functorial w.r.t. positive maps.

Section \ref{section:hecke_wgraph_algebras} recalls the definition of Iwahori-Hecke algebras, $W$-graphs and $W$-graph algebras.

Section \ref{section:induction} proves that Howlett-Yin induction is well-defined as a map $\Omega_J\textbf{-mod} \to \Omega\textbf{-mod}$. The proof is inspired by Lusztig's elegant treatment of the $\mu$-values in \cite{lusztig2003hecke} instead of the more laborious proof in Howlett and Yin's paper. The proof in the style of Lusztig has the additional bonus that it provides an algorithm to compute $p$-polynomials and $\mu$-values without having to compute $r$-polynomials as an intermediate step. Specifically, this is Algorithm \ref{algorithm:p_mu}. It is shown how this theorem recovers earlier results, including Howlett and Yin's. As an application it is proven that the $W$-graph algebra associated to a parabolic subgroup $W_J\leq W$ can be canonically identified with a subalgebra of the $W$-graph algebra of $W$.

Section \ref{section:categorial_properties} then proves that Howlett-Yin induction has many of the nice properties one expects: it is a indeed a functor between the two module categories, it can be represented as tensoring with a certain bimodule, it satisfies a transitivity property and an analogue of the Mackey-theorem.

Section \ref{section:applications} then applies these findings. An improved, more efficient algorithm to compute $\mu$-values is given which generalises ideas from Geck's PyCox software (see \cite{geck2012pycox}). Additionally a very short proof of a result of Geck on induction of Kazhdan-Lusztig cells (from \cite{geck2003induction}) is given.

\section{Canonicalisation of modules}\label{section:canonicalisation}

Fix a commutative ring $k$ and a totally ordered, abelian group $(\Gamma,+,\leq)$, i.e. $\leq$ is a total order on $\Gamma$ such that $x\leq y \implies x+z\leq y+z$ holds.

Consider the $k$-algebra $\mathcal{A}:=k[\Gamma]$. As is common when considering group algebras of additively written groups, we will denote the group element $\gamma\in\Gamma$ as $v^\gamma\in\mathcal{A}$ and think of $\mathcal{A}$ as the ring of \enquote{generalized Laurent polynomials in $v$} with coefficients in $k$. This $k$-algebra has an involutive automorphism $\overline{\phantom{m}}$ defined by $\overline{v^\gamma}:=v^{-\gamma}$.

We also consider the smash product $\widehat{\mathcal{A}}:=\mathcal{A}\rtimes \langle\iota\rangle$ where $\langle\iota\rangle$ is a cyclic group of order two acting as $\overline{\phantom{m}}$ on $\mathcal{A}$.\footnotemark An $\widehat{\mathcal{A}}$-module is therefore the same as an $\mathcal{A}$-module $M$ equipped with an $\overline{\phantom{m}}$-semilinear involution $\iota: M\to M$.
\footnotetext{
Remember that given any $k$-algebra $A$, monoid $G$, and any monoid homomorphism $\phi: G\to(\End(A),\circ)$ the algebra $A\rtimes_\phi G$ is defined as the $k$-algebra that has $A \otimes_k k[G]$ as its underlying $k$-module and extends the multiplication of $A$ and $k[G]$ via $(a\otimes g)\cdot(b\otimes h) := a\phi(g)(b) \otimes gh$. It is also denoted $A\rtimes G$ if the morphism $\phi$ is understood.

Saying that $V$ is a $A\rtimes_\phi G$ module is equivalent to saying that $V$ is an $A$-module and comes with an action of $G$ on $V$ such that $g(a\cdot v) = \phi(g)(a)\cdot gv$ holds.
}

\begin{definition}
Define $\mathcal{A}_{<0}$ as the $k$-submodule of $\mathcal{A}$ spanned by all $v^\gamma$ with $\gamma\geq 0$. Similarly define $\mathcal{A}_{\geq 0}, \mathcal{A}_{\leq 0}$ and $\mathcal{A}_{<0}$.
	
Note that $\mathcal{A}_{\geq 0}$ is a subalgebra of $\mathcal{A}$ and $\mathcal{A}_{>0}$ an ideal inside it.
\end{definition}

\begin{definition}
Let $M$ be an arbitrary $k$-module. The scalar extension $\mathcal{A}\otimes_k M$ is naturally an $\mathcal{A}$-module and via $\iota(a\otimes m) := \overline{a} \otimes m$ it is also an $\widehat{\mathcal{A}}$-module which will be denoted by $\widehat{M}$. Any $\widehat{\mathcal{A}}$-module $V$ that is isomorphic to $\widehat{M}$ for some $M\in k\textbf{-mod}$ is called a \emph{canonical module} and any $\widehat{\mathcal{A}}$-module isomorphism $c: \widehat{M} \to V$ is called a \emph{canonicalisation of $V$}.

If $M$ is free and $(b_x)_{x\in X}$ is a basis of $M$, then the image of $(1\otimes b_x)_{x\in X}$ under a canonicalisation $c$ is called the \emph{canonical basis of $V$} associated to $(b_x)_{x\in X}$ and $c$.
\end{definition}

\begin{remark}
Note that $M\mapsto \widehat{M}$ and $f\mapsto \id_\mathcal{A}\otimes f$ is a faithful functor $k\textbf{-mod} \to \widehat{\mathcal{A}}\textbf{-mod}$. It is left adjoint to the fixed point functor $V\mapsto\Set{x\in V | \iota(v) = v}$.
\end{remark}

\begin{example}
As we will see, the Kazhdan-Lusztig basis $(C_w)_{w\in W}$ of an Iwahori-Hecke algebra $H=H(W,S)$ is a canonical basis of the $\widehat{\IZ[v^{\pm1}]}$-module $V:=H$ where $\iota:H\to H$ is defined by $\iota(T_w)=T_{w^{-1}}^{-1}$.

The module $M$ is the $\IZ$-span of the \enquote{standard basis} $M:=\bigoplus_{w\in W} \IZ T_w$. Kazhdan and Lusztig's classical result that an $\iota$-invariant $\IZ[v^{\pm1}]$-basis of $H$ exists is now precisely the statement that $T_w\mapsto C_w$ defines a canonicalisation map. In this sense $(C_w)_{w\in W}$ is the canonical basis of $H$ associated to the standard basis.
\end{example}

\begin{example}
Canonical bases of quantum group representations in the sense of Lusztig and Kashiwara (c.f. \cite{lusztig1990canonical} and \cite{kashiwara1990crystalizing}) are examples of canonical basis in the sense of the definition above.

In Lusztig's notation $\textbf{B}\subset\mathcal{L}$ is a canonical basis of the $\widehat{\IZ[v^{\pm1}]}$-module $V:=\mathcal{L}+\overline{\mathcal{L}}$ which is the $\IZ[v^{\pm1}]$-span of any PBW-basis of $\textbf{U}^+$. The $\IZ$-module $M$ corresponds to $\mathcal{L}/v^{-1}\mathcal{L}$ (which is also isomorphic to $\mathcal{L}\cap\overline{\mathcal{L}}$ as well as $\overline{\mathcal{L}}/v\overline{\mathcal{L}}$) and the standard basis of $M$ is the image of any PBW-basis of $\textbf{U}^+$ in $\mathcal{L}/v^{-1}\mathcal{L}$.
\end{example}

\begin{example}
Obviously most $\widehat{\mathcal{A}}$-modules are not canonical. For example the only canonical $\widehat{\mathcal{A}}$-module that is finitely generated over $k$ is the zero module. Hence $V=k[i]=k[x]/(x^2+1)$ is not canonical where $v$ operates as multiplication by $i$ and $\iota$ operates as $i\mapsto -i$. Therefore the question arises how one can recognize if a given module is canonical and how one can find a canonicalisation.
\end{example}

An obvious restatement of the definition is the following:
\begin{proposition}
Let $V$ be an arbitrary $\widehat{\mathcal{A}}$-module. Then $V$ is canonical if and only if there exists $k$-submodule $M$ of $V$ such that
\begin{enumerate}
	\item $V = \bigoplus_{\gamma\in\Gamma} v^\gamma M$ as a $k$-module.
	\item $\iota$ operates as $1$ on $M$, i.e. $\iota\cdot m = m$ for all $m\in M$.
\end{enumerate}
In this case $c: \widehat{M} \to V, a\otimes m \mapsto am$ is a canonicalisation.
\end{proposition}

\begin{definition}
Let $V$ be an $\widehat{\mathcal{A}}$-module and $(X,\leq)$ a poset. A \emph{$X$-graded shadow of $V$} is a collection $(M_x)_{x\in X}$ of $k$-submodules of $V$ such that
\begin{enumerate}
	\item $V = \bigoplus_{\substack{x\in X \\ \gamma\in\Gamma}} v^\gamma M_x$ as a $k$-module and
	\item $\iota\cdot m_z \in m_z + \sum_{y<z} \mathcal{A}M_y$ for all $m_z\in M_z$
\end{enumerate}
\end{definition}

\begin{remark}
In light of the above proposition a shadow is something like a canonicalisation \enquote{up to lower order error terms}. Theorem \ref{canonicalisation_from_shadows} shows that these error terms can be corrected by a \enquote{triangular base change} if the poset satisfies a finiteness condition. It therefore provides a sufficient criterion for the existence of a canonicalisation which is inspired by the construction of the Kazhdan-Lusztig basis of Iwahori-Hecke algebras as well as similar constructions by Howlett and Yin, Deodhar, Geck and many more. Theorem \ref{canonicalisation_from_shadows} is  precisely the common thread in all these constructions.
\end{remark}

First we need a lemma.
\begin{lemma}\label{positive_symmetric_implies_zero}
Let $V$ be an $\widehat{\mathcal{A}}$-module, $(X,\leq)$ a partially ordered set and $(M_x)_{x\in X}$ a $X$-graded shadow of $V$. Furthermore let $f\in V$ be an arbitrary element with $f = \sum_{x\in X_0} f_x$ for some finite $X_0\subseteq X$ and $f_x \in\mathcal{A}_{>0} M_x$.

If $f$ satisfies $\iota\cdot f = f$, then $f_x = 0$ for all $x\in X_0$.
\end{lemma}
\begin{proof}
Assume the contrary. Wlog we can also assume $f_x\neq 0$. Otherwise we could just shrink the set $X_0$. Let $X_1\subseteq X_0$ be the subset of all maximal elements of $X_0$ and $X_2:=X_0\setminus X_1$. Thus
\[f \in \sum_{x\in X_1} \mathcal{A}_{>0} M_x + \sum_{x\in X_2} \mathcal{A}_{>0} M_x \subseteq \sum_{x\in X_1} \mathcal{A}_{>0} M_x + \sum_{x\in X_2} \sum_{y\leq x} \mathcal{A} M_y\]
Because $f_x\in\mathcal{A}_{>0} M_x$, it is a $\mathcal{A}_{>0}$-linear combination of elements of $M_x$, say $f_x = \sum_{i=1}^{n_x} a_{ix} m_{ix}$ for some $a_{ix}\in\mathcal{A}_{>0}$ and $m_{ix}\in M_x$.

Also note that $\iota$ maps every subspace of the form $\sum_{y\leq x} \mathcal{A} M_y$ into itself because $(M_x)_{x\in X}$ is an $X$-graded shadow. Thus
\begin{align*}
	\iota f &\in \sum_{x\in X_1, i=1..n_x} \iota(a_{ix} m_{ix}) + \iota(\sum_{x\in X_2} \sum_{y\leq x} \mathcal{A} M_y) \\
	&= \sum_{x\in X_1, i=1..n_x} \overline{a_{ix}} \iota(m_{ix}) + \sum_{x\in X_2} \sum_{y\leq x} \mathcal{A} M_y \\ 
	&= \sum_{x\in X_1, i=1..n_x} \overline{a_{ix}} m_{ix} + \sum_{x\in X_2} \sum_{y\leq x} \mathcal{A} M_y \\
	&\subseteq \sum_{x\in X_1} \mathcal{A}_{<0} M_x + \sum_{x\in X_2} \sum_{y\leq x} \mathcal{A} M_y
\end{align*}

$X_1$ is a non-empty subset because $X_0$ is non-empty and finite. 
Comparing the $x$-components of $f$ and $\iota f$ for $x\in X_1$ we find $f_x\in\mathcal{A}_{>0}M_x \cap \mathcal{A}_{<0}M_x = 0$ contrary to the assumption $f_x\neq 0$.
\end{proof}

\begin{theorem}\label{canonicalisation_from_shadows}
Let $(X,\leq)$ be a poset such that $(-\infty,y]:=\Set{x\in X | x\leq y}$ is finite for all $y\in X$.
\begin{enumerate}
	\item If an $\widehat{\mathcal{A}}$-module $V$ has a $X$-graded shadow $(M_x)_{x\in X}$ then it is canonical and there exists a unique canonicalisation $c: \widehat{M} \to V$ where $M := \bigoplus_{x\in X} M_x$ such that $c(1\otimes m)\in m+\mathcal{A}_{>0} M$ for all $m\in M$.
	
	More precisely it satisfies $c(1\otimes m_x) \in m_x + \sum_{y<x} \mathcal{A}_{>0} M_y$ for all $m_x\in M_x$.
	\item The canonicalisation above depends functorially on the shadow w.r.t. \emph{positive maps}. More precisely let $V_1,V_2$ be two $\widehat{\mathcal{A}}$-modules with $X$-graded shadows $(M_{i,x})_{x\in X}$ and canonicalizations $c_i: \widehat{M_i}\to V_i$ and let $\phi: V_1\to V_2$ be a $\widehat{\mathcal{A}}$-linear map with $\phi(\mathcal{A}_{\geq 0} M_1)\subseteq\mathcal{A}_{\geq 0} M_2$.
	
	There is an induced map $M_1 = \mathcal{A}_{\geq 0}M_1 / \mathcal{A}_{>0} M_1 \xrightarrow{\phi} \mathcal{A}_{\geq 0} M_2/\mathcal{A}_{>0} M_2 = M_2$ and this induces an $\widehat{\mathcal{A}}$-linear map $\widehat{\phi}: \widehat{M_1}\to\widehat{M_2}$. This map satisfies $c_2\circ \widehat{\phi} = \phi\circ c_1$ holds, i.e. the diagram in Figure \ref{fig:Canonicalisation} commutes.
	\begin{figure}[ht]
	\centering
	\begin{tikzpicture}
		\node (M1hat) at (0,2) {$\widehat{M_1}$};
		\node (M2hat) at (2,2) {$\widehat{M_2}$};
		\node (V1) at (0,0) {$V_1$};
		\node (V2) at (2,0) {$V_2$};
		
		\path[->,font=\scriptsize]
			(M1hat) edge node[above]{$\widehat{\phi}$} (M2hat)
			(V1) edge node[below]{$\phi$} (V2)
			(M1hat) edge node[left]{$c_1$} node[right]{$\isomorphic$} (V1)
			(M2hat) edge node[left]{$c_2$} node[right]{$\isomorphic$} (V2);

	\end{tikzpicture}
	\caption{Functoriality of canonicalisation of shadows}%
	\label{fig:Canonicalisation}%
	\end{figure}
\end{enumerate}
\end{theorem}
Before we begin the proof observe that any $\mathcal{A}$-linear map $f: \mathcal{A}M_z \to \mathcal{A}M_y$ is uniquely determined by its restriction $M_z \to \mathcal{A}M_y$ which is a $k$-linear map and can be written as $f(m_z)=\sum_{\gamma\in\Gamma} v^\gamma f_\gamma(m_z) $ with uniquely determined $k$-linear maps $f_\gamma: M_z\to M_y$ that have the property that $\Set{\gamma | f_\gamma(m_z)\neq 0}$ is finite for each $m_z\in M_z$ so that the sum is indeed well-defined. Having this way of writing these maps in mind we can define $\overline{f}:\mathcal{A}M_z\to \mathcal{A}M_y$ to be the $\mathcal{A}$-linear map with $\overline{f}(m_z):=\sum_{\gamma\in\IZ} v^{-\gamma} f_\gamma(m_z)$ for all $m_z\in M_z$. Note that $\overline{\phantom{f}}$ is compatible with composition, i.e. $\overline{f\circ g}=\overline{f}\circ\overline{g}$.

We will use this notation for the proof to simplify the notation.

\begin{proof}
The uniqueness of $c$ follows from the above lemma because if $c,c':\widehat{M}\to V$ are two canonicalisations satisfying the stated property then $f:=c(1\otimes m)-c'(1\otimes m)$ is an element of $\mathcal{A}_{>0}M$ with $\iota\cdot f= f$ so that $f=0$ by Lemma \ref{positive_symmetric_implies_zero}.

\medbreak
Concerning the existence consider the $\mathcal{A}$-linear maps $\rho_{yz}: \mathcal{A}M_z \to \mathcal{A}M_y$ defined by
\[\forall m_z\in M_z: \iota\cdot m_z = \sum_{y} \rho_{yz}(m_z).\]
By assumption $\rho_{yz}=0$ unless $y\leq z$ and $\rho_{zz}(m_z)=m_z$.

Following the usual construction of the Kazhdan-Lusztig polynomials and $R$-polynomials we will recursively construct $\mathcal{A}$-linear maps $\pi_{yz}: \mathcal{A}M_z \to \mathcal{A}_{\geq 0} M_y$ such that:
\begin{itemize}
	\item $\pi_{yz}=0$ unless $y\leq z$ and $\pi_{zz}(m_z)=m_z$
	\item $\pi_{xz} = \sum_{x\leq y\leq z} \rho_{xy} \circ \overline{\pi_{yz}}$
\end{itemize}

The first step is to observe
\begin{equation}
\sum_{x\leq y\leq z} \rho_{xy} \circ \overline{\rho_{yz}} = \begin{cases} \id_{\mathcal{A}M_x} & \text{if}\,x=z \\ 0 & \text{otherwise}\end{cases}
\label{eq:Canonicalisation_R1}
\end{equation}
This follows from the fact that $\iota$ has order two:
\begin{align*}
	m_z &= \iota\cdot\iota\cdot m_z \\
	&= \sum_{y\leq z} \iota \cdot\underbrace{\rho_{yz}(m_z)}_{\in \mathcal{A}M_y} \\
	&= \sum_{x\leq y\leq z} (\rho_{xy}\circ \overline{\rho_{yz}})(m_z)
\end{align*}

Fix $z\in X$. Define $\pi_{zz}(m_z):=m_z$ and $\pi_{xz}:=0$ for all $x\not\leq z$. If $x<z$ then assume inductively that $\pi_{yz}$ is already known for all $x<y\leq z$. It is therefore possible to define
\[\alpha_{xz} := \sum_{x<y\leq z} \rho_{xy} \circ \overline{\pi_{yz}}.\]
This map satisfies
\begin{align*}
	\alpha_{xz} &= \sum_{x<y\leq z} \rho_{xy} \circ \overline{\sum_{y\leq w\leq z} \rho_{yw} \circ \overline{\pi_{wz}}} \\
	&= \sum_{x<y\leq w\leq z} \rho_{xy} \circ \overline{\rho_{yw}} \circ \pi_{wz} \\
	&= \sum_{x<w\leq z} \left(\sum_{x<y\leq w} \rho_{xy} \circ \overline{\rho_{yw}} \right) \circ \pi_{wz} \\
	&= \sum_{x<w\leq z} \left(0 - \rho_{xx}\circ\overline{\rho_{xw}} \right) \circ \pi_{wz} \\
	&= \sum_{x<w\leq z} -\overline{\rho_{xw}} \circ \pi_{wz} \\
	&= -\overline{\alpha_{xz}}
\end{align*}
Therefore we obtain $\alpha_0=0$ in the decomposition $\alpha_{xz} = \sum_{\gamma\in\Gamma} v^\gamma \alpha_\gamma$. Now define $\pi_{xz} := \sum_{\gamma>0} v^\gamma \alpha_\gamma$ so that $\alpha_{xz} = \pi_{xz} - \overline{\pi_{xz}}$ holds. This shows that $\pi_{xz}(m_z)\in \mathcal{A}_{>0} M_x$ as well as
\[\sum_{x\leq y\leq z} \rho_{xy} \circ \overline{\pi_{yz}} = \rho_{xx} \circ \overline{\pi_{xz}} + \alpha_{xz} = \overline{\pi_{xz}} + \alpha_{xz} = \pi_{xz}.\]
Thus the existence of all $\pi_{xz}$ is established and we can define the $\mathcal{A}$-linear map $c: \widehat{M} \to V$ by
\[\forall m_z\in M_z: c(1\otimes m_z) := \sum_{x\leq z} \pi_{xz}(m_z) = m_z + \sum_{x<z} \pi_{xz}(m_z).\]
It is bijective because it is "upper triangular with unit diagonal". The map is also $\widehat{\mathcal{A}}$-linear because
\begin{align*}
	\iota \cdot c(1\otimes m_z) &= \sum_{y\leq z} \iota\cdot \pi_{yz}(m_z) \\
	&= \sum_{x\leq y\leq z} \rho_{xy} \overline{\pi_{yz}}(m_z) \\
	&= \sum_{x\leq z} \pi_{xz}(m_z) \\
	&= c(1\otimes m_z) \\
	&= c(\iota\cdot(1\otimes m_z))
\end{align*}

\bigbreak
Finally we have to show that $c$ is functorial. Let $M_1, M_2, \phi$ be as in the statement of the theorem and fix an arbitrary $m_1\in M_1$. Then $c_1(m_1) \in  m_1 + \mathcal{A}_{>0} M_1$ so that $\phi(c_1(m_1)) \in \phi(m_1) + \mathcal{A}_{>0} M_2$. Also $\widehat{\phi}(m_1)\in \phi(m_1)+\mathcal{A}_{>0}M_2$ by construction of $\widehat{\phi}$ so that $c_2(\widehat{\phi}(m_1))\in \widehat{\phi}(m_1)+\mathcal{A}_{>0}M_2 = \phi(m_1)+\mathcal{A}_{>0}M_2$. Therefore $f:=\phi(c_1(m_1)) - c_2(\widehat{\phi}(m_1))\in\mathcal{A}_{>0}M_2$. Additionally, since all four maps are $\widehat{\mathcal{A}}$-linear and $m_1\in\widehat{M_1}$ is $\iota$-invariant $f$ satisfies $\iota f = f$ so that Lemma \ref{positive_symmetric_implies_zero} implies $f=0$. This proves the commutativity of the diagram.
\end{proof}

\begin{corollary}
Let $(X,\leq)$ be a poset such that $\Set{x\in X | x\leq y}$ is finite for all $y\in X$. Furthermore let $V$ be an $\widehat{\mathcal{A}}$-module, $U$ an $\widehat{\mathcal{A}}$-submodule of $V$ and $(M_x)_{x\in X}$ an $X$-graded shadow for $V$. Define $N_x := U\cap M_x$ for all $x\in X$.

If $U$ is generated as an $\mathcal{A}$-module by $\sum_{x\in X} N_x$, then $(N_x)$ is an $X$-graded shadow for $U$, $(M_x/N_x)$ is an $X$-graded shadow for $V/U$, the canonicalisation $\widehat{M}\to V$ restricts to the canonicalisation $\widehat{N} \to U$ and induces the canonicalisation $\widehat{M / N} \to V/U$ on the quotients.
\end{corollary}
\begin{proof}
This follows immediately from functoriality of canonicalisation applied to the embedding $U\hookrightarrow V$ and the quotient map $V\to V/U$ respectively.
\end{proof}

\begin{remark}
In terms of canonical bases this corollary recovers the theorem that if $(t_x)_{x\in X}$ is an $\mathcal{A}$-basis for $V$ and $U$ is spanned as an $\mathcal{A}$-module by a subset $(t_x)_{x\in Y}$ of that basis, then the canonical basis for $U$ is the subset $(c_x)_{x\in Y}$ of the canonical basis $c_x:=c(t_x)$ of $V$ and the canonical basis of the quotient $V/U$ is spanned by the vectors $(c_x)_{x\in X\setminus Y}$ (more precisely by their images under the quotient map $V\to V/U$).
\end{remark}
\begin{remark}
An important special case of this corollary is the case where $U$ is of the form $U=\sum_{x\in I} \mathcal{A}M_x$ for some order ideal $I\trianglelefteq X$ (i.e. a subset with the property $x\in I \wedge y\leq x \implies y\in I$). Note that all such $U$ are $\widehat{\mathcal{A}}$-submodules by definition of $X$-graded shadows.
\end{remark}

\section{Hecke algebras, \texorpdfstring{$W$}{W}-graphs and \texorpdfstring{$W$}{W}-graph algebras}\label{section:hecke_wgraph_algebras}

For the rest of the paper fix a (not necessarily finite) Coxeter group $(W,S)$, a totally ordered, additive group $\Gamma$ (which soon will be further restricted to be $\IZ$) and a weight function $L:W\to\Gamma$, i.e. a function with $l(xy)=l(x)+l(y) \implies L(xy)=L(x)+L(y)$. We will use the shorthand $v_s := v^{L(s)}\in\IZ[\Gamma]$ and the standard assumption $L(s)>0$ for all $s\in S$.

\begin{definition}[c.f. \cite{geckjacon}]
The \emph{Iwahori-Hecke algebra} $H=H(W,S,L)$ is the $\mathbb{Z}[\Gamma]$-algebra  which is freely generated by $(T_s)_{s\in S}$ subject only to the relations
\[T_s^2 = 1 + (v_s-v_s^{-1}) T_s\quad\textrm{and}\]
\[\underbrace{T_s T_t T_s \ldots}_{m_{st}\,\text{factors}} = \underbrace{T_t T_s T_t \ldots}_{m_{st}\,\text{factors}}\]
where $m_{st}$ denotes the order of $st\in W$.

\medbreak
Because of the braid relations and Matsumoto's theorem (c.f. \cite{matsumoto1964}), we can define the \emph{standard basis elements} as
\[T_w := T_{s_1} T_{s_2} \cdots T_{s_l}\]
where $w=s_1 s_2 \cdots s_l$ is any reduced expression of $w\in W$ in the generators. Note that $T_1=1$.

\medbreak
For each parabolic subgroup $W_J\leq W$ the Hecke algebra $H(W_J,J,L_{|W_J})$ will be identified with the \emph{parabolic subalgebra} $H_J := \operatorname{span}_{\IZ[\Gamma]}\Set{T_w | w\in W_J}\subseteq H$.
\end{definition}

\begin{definition}[c.f. \cite{KL} and \cite{geckjacon}]
Let $k$ be a commutative ring.
A \emph{$W$-graph with edge weights in $k$} is a triple $(\mathfrak{C},I,m)$ consisting of a finite set $\mathfrak{C}$ of \emph{vertices}, a \emph{vertex labelling} map $I: \mathfrak{C}\to \{J \mid J\subseteq S\} $ and a family of \emph{edge weight} matrices $m^s\in k^{\mathfrak{C}\times\mathfrak{C}}$ for $s\in S$ such that the following conditions hold:
\begin{enumerate}
	\item $\forall x,y\in\mathfrak{C}: m_{xy}^s \neq 0 \implies s\in I(x)\setminus I(y)$.
	\item The matrices
	\[\omega(T_s)_{xy} := \begin{cases} -v_s^{-1} & \textrm{if}\;x=y, s\in I(x) \\ v_s & \textrm{if}\;x=y, s\notin I(x) \\ m_{xy}^s & \textrm{otherwise}\end{cases}\]
	induce a matrix representation $\omega: k[v^{\pm1}]H\to k[v^{\pm1}]^{\mathfrak{C}\times\mathfrak{C}}$.
\end{enumerate}
The associated directed graph is defined as follows: The vertex set is $\mathfrak{C}$ and there is a directed edge $x\leftarrow y$ if and only if $m_{xy}^s\neq 0$ for some $s\in S$. If this is the case, then the value $m_{xy}^s$ is called a \emph{weight} of the edge. The set $I(x)$ is called the \emph{vertex label} of $x$.
\end{definition}

\begin{remark}
In the equal-parameter case (i.e. $\Gamma=\IZ$ and $L(s)=1$ for all $s\in S$) one can show $m_{xy}^s = m_{xy}^t$ for all $s,t\in I(x)\setminus I(y)$ so that it is well-defined to speak of \emph{the} weight of the edge $x\leftarrow y$.

This does no longer hold in the multi-parameter case, so that one could say that the edges have a tuple of weights attached to them.
\end{remark}

\begin{remark}
Note that condition a. and the definition of $\omega(T_s)$ already guarantee $\omega(T_s)^2=1+(v_s-v_S^{-1})\omega(T_s)$ so that the only non-trivial requirement in condition b. is the braid relation $\omega(T_s)\omega(T_t)\omega(T_s)\ldots = \omega(T_t)\omega(T_s)\omega(T_t)\ldots$.
\end{remark}

Given a $W$-graph as above the matrix representation $\omega$ turns $k[\Gamma]^{\mathfrak{C}}$ into a module for the Hecke algebra. It is natural to ask whether a converse is true. In many situations the answer is yes as shown by Gyoja.
\begin{theorem}[c.f. \cite{Gyoja}]
Let $W$ be finite, $K\subseteq\mathbb{C}$ be a splitting field for $W$ and assume $\Gamma=\IZ$ and $L(s)=1$ for all $s\in S$. Then every irreducible representation of $K(v)H$ can be realized as a $W$-graph module for some $W$-graph with edge weights in $K$.
\end{theorem}
\begin{remark}
The same is true in the multi-parameter case if Lusztig's conjecture \textbf{P15} or similar properties like Geck and Jacon's $(\spadesuit)$ and $(\clubsuit)$ hold for $(W,S,L)$, see \citep[2.7.12]{geckjacon} or \citep[4.3.5]{hahn2013diss} for a proof.
\end{remark}
\begin{remark}
Gyoja also provides an example of a finite-dimensional representation of the affine Weyl group of type $\tilde{A}_n$ that is not induced by a $W$-graph.
\end{remark}

\begin{convention}
For the remainder of the paper we will assume $\Gamma=\IZ$ (although we will still write $\Gamma$ when referring to the group of exponents of the Laurent polynomials).

It is not strictly speaking necessary to do this since the results also hold in the general case, but the general definitions and proofs are much more technical because one has to work with infinite series of the form $\sum_{-L(s)<\gamma<L(s)} x_{s,\gamma}v^\gamma$ and must ensure their convergence in the appropriate sense in all proofs.

By restricting to $\Gamma=\IZ$ all the relevant sums become finite sums and separate convergence arguments are unnecessary.
\end{convention}

\begin{definition}[The $W$-graph algebra]
Assume $\Gamma=\IZ$ and consider the free algebra $\IZ\langle e_s, x_{s,\gamma} | s\in S, -L(s)<\gamma<L(s)\rangle$. Define
\[j(T_s) := -v_s^{-1} e_s + v_s (1-e_s) + \sum_{-L(s)<\gamma<L(s)} v^\gamma x_{s,\gamma}  \in \mathbb{Z}[\Gamma] \otimes_\IZ \IZ\langle e_s, x_{s,\gamma} \rangle\]
for all $s,t\in S$ and write
\[\sum_{\gamma\in\Gamma}  v^\gamma \otimes y^\gamma(s,t) = \underbrace{j(T_s)j(T_t)j(T_s)\ldots}_{m_{st}\,\text{factors}} - \underbrace{j(T_t)j(T_s)j(T_t)\ldots}_{m_{st}\,\text{factors}}\]
for some $y^\gamma(s,t)\in\IZ\langle e_s, x_{s,\gamma} \rangle$.

\medbreak
Define $\Omega$ to be the quotient of $\IZ\langle e_s, x_{s,\gamma} \rangle$ modulo the relations
\begin{enumerate}
	\item $e_s^2=e_s$, $e_s e_t = e_t e_s$,
	\item $e_s x_{s,\gamma}=x_{s,\gamma}$, $x_{s,\gamma} e_s = 0$,
	\item $x_{s,\gamma} = x_{s,-\gamma}$ and
	\item $y^\gamma(s,t) = 0$
\end{enumerate}
for all $s,t\in S$ and all $\gamma\in\Gamma$.

Finally define the element
\[x_s := \sum_{\gamma\in\Gamma} v^\gamma x_{s,\gamma} \in \IZ[\Gamma]\Omega.\]
\end{definition}

\begin{remark}
The definition immediately implies that $T_s\mapsto j(T_s)$ defines a homomorphism of $\IZ[\Gamma]$-algebras $j: H\to\IZ[\Gamma]\Omega$. In fact this is an embedding as shown in \citep[Corollary 10]{hahn2014wgraphs1}. We will identify $H$ with its image in $\IZ[\Gamma]\Omega$ and suppress any mention of $j$ from now on to simplify the notation.

Note that $C_s = T_s - v_s = -(v_s+v_s^{-1})e_s + x_s$.
\end{remark}

\begin{remark}\label{w_graph_algebra:H_vs_Omega_modules}
$W$-graph algebras have the distinguishing feature that each $W$-graph $(\mathfrak{C},I,m)$ with edge weights in $k$ not only defines the structure of an $H$-module on $k[\Gamma]^{\mathfrak{C}}$ but that it induces a canonical $k\Omega$-module structure on $k^\mathfrak{C}$ via
\[e_s \cdot \mathfrak{z} := \begin{cases} \mathfrak{z} & s\in I(\mathfrak{z}) \\ 0 & \text{otherwise}\end{cases}\]
\[x_s \cdot \mathfrak{z} := \sum_{\mathfrak{x}\in\mathfrak{C}} m_{\mathfrak{xz}}^s \mathfrak{x}\]
for all $\mathfrak{z}\in\mathfrak{C}$. Then $k[\Gamma]^{\mathfrak{C}\times\mathfrak{C}}$ is a $k[\Gamma]\Omega$-module and restriction to a $H$-module gives back the $H$-module in the definition.

Conversely if $V$ is a $k\Omega$-module that has a $k$-basis $\mathfrak{C}$ w.r.t. which all $e_s$ act as diagonal matrices, then $V$ is obtained from a $W$-graph $(\mathfrak{C},I,m)$ in this way: $m^s$ is the matrix representing the action of $x_s$ and $I(\mathfrak{z}) = \Set{s\in S | e_s \mathfrak{z}=\mathfrak{z}}$. In this way one can interpret $\Omega$-modules as $W$-graphs up to choice of a basis. (See \citep[Theorem 9]{hahn2014wgraphs1} or \citep[4.2.18]{hahn2013diss} for a detailed proof of these claims.)

Of course $V$ does not need to have a basis at all if $k$ is not a field so that $k\Omega$-modules are indeed more general than $W$-graphs.
\end{remark}

\begin{remark}
Every $k\Omega$-module $V$ is also a $k$-module and gives us a canonical $\widehat{k[\Gamma]}$-module $\widehat{V}=k[\Gamma]\otimes_k V$ which is also an $k[\Gamma]H$-module by restriction along $k[\Gamma]H\hookrightarrow k[\Gamma]\Omega$.

The results about Howlett-Yin induction will all be proved via canonicalisation of $k[\Gamma]H$-modules on which we will define the appropriate $k\Omega$-module structure.
\end{remark}

\begin{example}
The trivial group is a Coxeter group $(1,\emptyset)$ and its associated $W$-graph algebra is just $\IZ$.

A cyclic group of order 2 is a Coxeter group $(\set{1,s},\set{s})$ of rank 1 and its associated $W$-graph algebra is as a free $\IZ$-module with basis $\Set{e_s, 1-e_s}\cup\Set{x_{s,\gamma} | 0\leq\gamma<L(s)}$. The multiplication of the basis elements is completely determined by the relations because $x_{s,\gamma_1} x_{s,\gamma_2} = x_{s,\gamma_1} (e_s x_{s,\gamma_2}) = (x_{s,\gamma_1} e_s) x_{s,\gamma_2} = 0$.
\end{example}
\section{Howlett-Yin induction}\label{section:induction}

Let $M$ be any $k\Omega$-module. Then $k[\Gamma]M$ is naturally a $k[\Gamma]k\Omega$-module and by restriction of scalars it is also a $k[\Gamma]H$-module which we will (somewhat abusing the notation) denote by $\Res_H^\Omega M$.

Given any $\Omega_J$-module $M$, its restriction to $H_J$ can be induced to a $H$-module. The goal of this subsection is to prove that $\Ind_{H_J}^H \Res_{H_J}^{\Omega_J} M$ not only has the structure of an $\Omega$-module but that this module structure can be chosen functorially in $M$. The specific construction of this functor is a generalisation of such a construction Howlett and Yin gave in the equal-parameter case and for the special case that $M$ is given by $W$-graph. We will prove it in the general case using idea's by Lusztig (see \cite{lusztig2003hecke}).

\subsection{Preparations}

\begin{proposition}\label{eigenspace_Ts}
Let $M$ be a $k\Omega$-module and $a\in\Res_H^\Omega(M)$ be an arbitrary element. Then the following holds for all $s\in T$:
$T_s a = -v_s^{-1} a \iff e_s a = a$
\end{proposition}
\begin{proof}
The forward implication can be seen as follows
\begin{align*}
	-v_s^{-1} a &= (-v_s^{-1} e_s + v_s (1-e_s) +x_s) a \\
	\implies 0 &= ((v_s^{-1}+v_s)(1-e_s) + x_s) a \\
	\implies 0 &= ((v_s^{-1}+v_s)(1-e_s)^2 + \smash{\underbrace{(1-e_s)x_s}_{=0}}) a \\
	&= (v_s^{-1}+v_s)(1-e_s)a \\
	\implies 0 &= (1-e_s)a
\end{align*}
where we used in the last step that $\Res_H^\Omega(M) = k[\Gamma]\otimes_k M$ as $k[\Gamma]$-modules to cancel $v_s^{-1}+v_s$. The backward implication is trivial: $T_s a = T_s e_s a = -v_s^{-1} e_s^2 a = -v_s^{-1} a$.
\end{proof}

We will need the following well-known facts about cosets of parabolic subgroups:
\begin{lemmadef}\label{lemma:deodhar}
Let $J\subseteq S$ be any subset and $W_J$ the associated parabolic subgroup. Then the following hold:
\begin{enumerate}
	\item $D_J := \Set{x\in W | \forall s\in J : l(xs)>l(x)}$ is a set of representatives for the left cosets of $W_J$ in $W$. Its elements are exactly the unique elements of minimal length in each coset. They have the property $l(xw)=l(x)+l(w)$ for all $w\in W_J$.
	\item Deodhar's Lemma (c.f. \cite{deodhar1977bruhat})
	
	For all $w\in D_J$ and all $s\in S$ exactly one of the following cases happens:
	\begin{enumerate}
		\item $sw>w$ and $sw\in D_J$
		\item $sw>w$ and $sw\notin D_J$. In this case $s^w:=w^{-1}sw\in J$.
		\item $sw<w$. In this case $sw\in D_J$ holds automatically.
	\end{enumerate}
	Thus for fixed $s\in S$ there is a partition $D_J=D_{J,s}^+ \sqcup D_{J,s}^0 \sqcup D_{J,s}^-$ and similarly for fixed $w\in D_J$ there is a partition $S=D_J^+(w) \sqcup D_J^0(w) \sqcup D_J^-(w)$ where
	\begin{alignat*}{2}
	D_{J,s}^+ &:= \Set{w | sw>w, sw\in D_J}    & \quad D_{J}^+(w) &:= \Set{s | sw>w, sw\in D_J} \\
 	D_{J,s}^0 &:= \Set{w | sw>w, sw\notin D_J} & \quad D_{J}^0(w) &:= \Set{s | sw>w, sw\notin D_J} \\
	D_{J,s}^- &:= \Set{w | sw<w}               & \quad D_{J}^-(w) &:= \Set{s | sw<w}
	\end{alignat*}
\end{enumerate}
If $J\subseteq K\subseteq S$ is another subset, then furthermore
\begin{enumerate}[resume]
	\item $D_J^K:=D_J\cap W_K$ is the set of distinguished left coset representatives for $W_J$ in $W_K$ and $D_K^S \times D_J^K \to D_J^S, (x,y)\mapsto xy$ is a length-preserving bijection.
	\item If $x\in D_K^S$ and $y\in D_J^K$, then
	\begin{align*}
		D_J^+(xy) &= \Set{s\in D_K^0(x) | s^x \in D_J^+(y)} \cup  D_K^+(x) \\
		D_J^0(xy) &= \Set{s\in D_K^0(x) | s^x \in D_J^0(y)} \\
		D_J^+(xy) &= \Set{s\in D_K^0(x) | s^x \in D_J^-(y)} \cup D_K^-(x)
	\end{align*}
\end{enumerate}
\end{lemmadef}
\begin{proof}
See \cite{geckpfeiffer} for example.
\end{proof}

To apply the previous observations about canonical modules we need a semilinear map on our modules.
\begin{lemmadef}\label{lemmadef:iota_p}
If $M$ is any $k\Omega$-module, then we will fix the notation $\iota$ for the canonical $\overline{\phantom{m}}$-semilinear map $a\otimes m\mapsto \overline{a}\otimes m$ on $k[\Gamma]\otimes_k M$.

Then the following hold:
\begin{enumerate}
	\item In the special case of $M=k\Omega$ itself, $\iota$ is ring automorphism of $k[\Gamma]\Omega$ with $\iota(T_s) = T_s^{-1} = T_s - (v_s-v_s^{-1})$. In particular $\iota$ restricts to a $\overline{\phantom{m}}$-semilinear involution of $H$.
	\item For general $M$, furthermore $\iota(ax) = \iota(a)\iota(x)$ holds for all $a\in k[\Gamma]\Omega$, $x\in k[\Gamma]M$.
\end{enumerate}
Now let $M$ be a $k\Omega_J$-module and $V:=\Res_{H_J}^{\Omega_J}(M)$ its associated $k[\Gamma]H_J$-module.
\begin{enumerate}[resume]
	\item $\iota(h\otimes x):=\iota(h)\otimes\iota(x)$ is a well-defined $\overline{\phantom{m}}$-semilinear involution on $\Ind_{H_J}^H(V)=H \otimes_{H_J} V$.
	\item The $k$-submodules $V_w := \Set{T_w \otimes m | m\in M}\subseteq\Ind_{H_J}^H(V)$ for $w\in D_J$ constitute a $D_J$-graded shadow on $\Ind_{H_J}^H(V)$ where $D_J$ is endowed with the Bruhat-Chevalley-order.
	\item The maps $\rho_{xz}:V_z\to k[\Gamma]V_x$ and $\pi_{xz}:V_z\to k[\Gamma]V_x$ in Theorem \ref{canonicalisation_from_shadows} are of the form
	\[\rho_{xz}(T_z\otimes m) = T_x\otimes r_{x,z} m \quad\text{and}\quad \pi_{xz}(T_z\otimes m) = T_x\otimes p_{x,z} m\]
	for elements $r_{x,z} \in \IZ[\Gamma]\Omega_J$, $p_{x,z}\in \IZ[\Gamma_{\geq 0}]\Omega_J$ that are independent of $M$.
\end{enumerate}
\end{lemmadef}
\begin{proof}
a.+b. The statements follow directly from the definition.

\medbreak
c. For all $a\in H_J$ one has $\iota(ha)\otimes\iota(x) = \iota(h)\iota(a)\otimes\iota(x) = \iota(h)\otimes\iota(a)\iota(x) = \iota(h)\otimes\iota(ax)$ which proves the well-definedness of $h\otimes x \mapsto \iota(h)\otimes\iota(x)$.

\medbreak
d. Note that
\[\iota(T_z\otimes m) = \sum_{w\in W} R_{w,z} T_w\otimes m = \sum_{x\in D_J} T_x \otimes \underbrace{\sum_{y\in W_J} R_{xy,z} T_y}_{=:r_{x,z}} m\]
where $R_{w,z}\in\IZ[\Gamma]$ are the Kazhdan-Lusztig $R$-polynomials, i.e. the polynomials defined by $\iota(T_z) = \sum_{w\in W} R_{w,z} T_w$.

Now note that $R_{xy, z} \neq 0$ implies $xy\leq z$ so that $x\leq xy\leq z$ and thus the summation only runs over $x$ with $x\leq z$. If furthermore $x=z$ then $xy\leq z$ can only be true if $y=1$. But we know $R_{z,z} = 1$. Therefore $(V_w)$ really is a $D_J$-graded shadow of $V$.

\medbreak
e. The claimed property for $\rho_{xz}$ follows from the equation above. The analogous property for $\pi_{xz}$ follows from the recursive construction of the $\pi_{xz}$ in Theorem \ref{canonicalisation_from_shadows}.
\end{proof}

\begin{lemmadef}\label{lemma:def_mu}
There is a unique family $\mu_{x,z}^s \in \IZ[\Gamma]\Omega_J$ for $x,z\in D_J, s\in S$ such that the following properties hold
\begin{enumerate}
	\item $\mu_{x,z}^s=0$ unless $x<z$, $z\in D_{J,s}^+ \cup D_{J,s}^0$, and $x\in D_{J,s}^0 \cup D_{J,s}^-$ hold.
	\item $\overline{\mu_{x,z}^s} = \mu_{x,z}^s$
	\item If $z\in D_{J,s}^+ \cup D_{J,s}^0$ and $x\in D_{J,s}^0 \cup D_{J,s}^-$, then 
	\[\mu_{x,z}^s + R + \sum_{x<y<z} p_{x,y}\mu_{y,z}^s \in \IZ[\Gamma_{>0}]\Omega_J\]
	where
	\[R = \begin{cases}
		- C_{s^x} p_{x,z}                  & z\in D_{J,s}^+,\; x\in D_{J,s}^0 \\
		v_s^{-1} p_{x,z}                   & z\in D_{J,s}^+,\; x\in D_{J,s}^- \\
		p_{x,z} C_{s^z} - C_{s^x} p_{x,z}  & z\in D_{J,s}^0,\; x\in D_{J,s}^0 \\
		p_{x,z} C_{s^z} + v_s^{-1} p_{x,z} & z\in D_{J,s}^0,\; x\in D_{J,s}^-
		\end{cases}\]
\end{enumerate}
These elements satisfy:
\begin{enumerate}[resume]
	\item $v_s \mu_{x,z}^s \in \IZ[\Gamma_{>0}]\Omega_J$.
\end{enumerate}
\end{lemmadef}
\begin{proof}
Conditions a.), b.) and c.) are precisely designed to give a recursive definition of $\mu_{x,z}^s$: The recursion happens along the poset $\Set{(x,z)\in D_J\times D_J | x\leq z}$ with the order $(x,z)\sqsubset(x',z') :\iff z<z' \vee (z=z' \wedge x>x')$.

First note that this is a well-founded poset since intervals in the Bruhat-Chevalley order are finite so that no infinite descending chain can exists and recursive definitions really make sense.

Now if $\mu_{x,z}^s$ is known for all $(x,z)\sqsubset(x',z')$ then c. determines the nonpositive part of $\mu_{x',z'}^s$ and by the symmetry condition b. $\mu_{x',z'}^s$ is completely determined. This shows how to define $\mu_{x',z'}$ for $x'\leq z'$.

\medbreak
The last property of $\mu_{x,z}^s$ follows by induction from this recursive construction. Note that $v_s R \in\IZ[\Gamma_{>0}]\Omega_J$ in all four cases because $p_{x,z}\in\IZ[\Gamma_{>0}]\Omega_J$ for all $x<z$. Assuming that $v_s \mu_{y,z}^s\in\IZ[\Gamma_{>0}]\Omega_J$ already holds for all $(y,z)\sqsubset(x,z)$ we find that
\[v_s \mu_{x,z}^s \equiv -v_s R-\sum_{x<y<z} p_{x,y} v_s \mu_{y,z}^s \equiv 0 \mod\IZ[\Gamma_{>0}]\Omega_J\]
holds.
\end{proof}

\subsection{Induction is well-defined}

\newcommand{\HYI}{\operatorname{HY}}
\begin{theorem}\label{H_induction_as_Omega_module}
Let $M$ be a $k\Omega_J$-module and $\HYI_J^S(M)$ the $k$-module $\bigoplus_{w\in D_J} M$. Denote elements of the $w$-component of $\HYI_J^S(M)$ as $w|m$. Further write $\mu_{x,z}^s$ as $\sum_{-L(s)<\gamma<L(s)} \mu_{x,z}^{s,\gamma} \cdot v^\gamma$ with $\mu_{x,z}^{s,\gamma}\in\Omega_J$.

With this notation $\HYI_J^S(M)$ becomes a $k\Omega$-module via
\[e_s \cdot z|m := \begin{cases} 0 & z\in D_{J,s}^+ \\ z| e_{s^z} m & z\in D_{J,s}^0 \\ z|m & z\in D_{J,s}^- \end{cases}\]
\[x_{s,\gamma} \cdot z|m := \begin{cases}
	\sum_{x<z} x|\mu_{x,z}^{s,\gamma} m + sz|m & z\in D_{J,s}^+, \gamma=0 \\
	\sum_{x<z} x|\mu_{x,z}^{s,\gamma} m  & z\in D_{J,s}^+,\gamma\neq 0 \\
	\sum_{x<z} x|\mu_{x,z}^{s,\gamma} m + z|x_{s^z,\gamma} m & z\in D_{J,s}^0 \\ 
	0 & z\in D_{J,s}^-
	\end{cases}\]
and the canonicalisation $c_M: \Res_H^\Omega \HYI_J^S(M) \to \Ind_{H_J}^H \Res_{H_J}^{\Omega_J}(M), z|m \mapsto \sum_{y} T_y\otimes p_{y,z} m$ is $k[\Gamma]H$-linear.
\end{theorem}

\begin{remark}
Comparing with \citep[Theorem 5.3]{howlett2003inducingI}, this theorem gives a more general result. It includes \citep[Theorem 5.3]{howlett2003inducingI}, as we will see in Proposition \ref{recovering_howlett_yin_result}, but it also encompasses the multi-parameter case (which Geck considered in \cite{geck2003induction} for the special case that $M$ is a left cell module).

We will also see in Theorem \ref{functoriality_HYI} that the construction is functorial in the appropriate sense which is also not included in Howlett and Yin's theorem. This more general, more abstract way of looking at induction will allow us to prove transitivity which was not included in Howlett and Yin's paper. It also allows for simplification of many known results.
\end{remark}

\begin{proof}[Proof of the main theorem]
We want to define a representation $\omega: \Omega \to \End(\HYI_J^S(M))$ and we already have a definition of $\omega(e_s)$ and $\omega(x_{s,y})$. Extend these $k[\Gamma]$-linearly to
\[\omega:\textrm{span}_{k[\Gamma]}\Set{1,e_s,x_{s,\gamma} | s\in S, -L(s)<\gamma<L(s)} \to k[\Gamma]\End(\HYI_J^S(M))\]

We need to show that $\omega$ satisfies all the relations of $\Omega$, that is
\begin{enumerate}
	\item $\omega(e_s)^2=\omega(e_s),\quad \omega(e_s)\omega(e_t) = \omega(e_t)\omega(e_s)$,
	\item $\omega(e_s)\omega(x_{s,\gamma}) = \omega(x_{s,\gamma}),\quad \omega(x_{s,\gamma})\omega(e_s) = 0$,
	\item $\omega(x_{s,\gamma}) = \omega(x_{s,-\gamma})$ and
	\item $\underbrace{\omega(T_s)\omega(T_t)\omega(T_s) \ldots}_{m_{st}\,\text{factors}} = \underbrace{\omega(T_t)\omega(T_s)\omega(T_t) \ldots}_{m_{st}\,\text{factors}}$ as an equation in $k[\Gamma] \End(\HYI_J (M))$ where $m_{st}:=\ord(st)$.
\end{enumerate}

\medbreak
The equations $\omega(e_s)^2 = \omega(e_s)$, $\omega(e_s) \omega(e_t) = \omega(e_t) \omega(e_s)$ and $\omega(x_{s,\gamma})=\omega(x_{s,-\gamma})$ follow directly from the definitions and the properties of $\mu$.

\medbreak
To prove that $\omega(T_s)$ satisfies the braid relations, we use the $k[\Gamma]$-linear bijection $c$ and show $c(\omega(T_s)z|m) = T_s c(z|m)$ for all $z\in D_J$ and all $m\in M$. Since $\Ind_{H_J}^H \Res_{H_J}^{\Omega_J}(M)$ is a $k[\Gamma]H$-module, the braid relations hold on the right hand side and will therefore also hold on the left hand side. Because of the equality $C_s=T_s-v_s$ this is equivalent to showing $c(\omega(C_s)w|m) = C_s c(w|m)$. We compare these two elements of $\Ind_{H_J}^H \Res_{H_J}^{\Omega_J} M$:

One the left hand side we find:
\begin{align*}
	c(\omega(C_s)z|m) &= c((-(v_s+v_s^{-1})e_s + x_s)\cdot z|m) \\
	&=\begin{cases}
		c(sz|m + \sum\limits_{y<z} y|\mu_{y,z}^s m) & z\in D_{J,s}^+ \\
		c(z|(-(v_s + v_s^{-1}) e_{s^z} + x_{s^z})m + \sum\limits_{y<z} y|\mu_{y,z}^s m) & z\in D_{J,s}^0 \\
		-(v_s^{-1}+v_s) c(z|m) & z\in D_{J,s}^-
	\end{cases} \\
	&=\begin{cases}
		c(v_s z|m + sz|m + \sum\limits_{y<z} y|\mu_{y,z}^s m) & z\in D_{J,s}^+ \\
		c(z|C_{s^z} m + \sum\limits_{y<z} y|\mu_{y,z}^s m) & z\in D_{J,s}^0 \\
		- (v_s+v_s^{-1}) c(z|m) & z\in D_{J,s}^-
	\end{cases} \\
	&= \begin{cases}
		\sum\limits_{x\in D_{J,s}} T_x \otimes \left(p_{x,sz}m + \sum\limits_{x\leq y<z} p_{x,y}\mu_{y,z}^s m\right) & z\in D_{J,s}^+ \\
		\sum\limits_{x\in D_{J,s}} T_x \otimes \left(p_{x,z}C_{s^z} m + \sum\limits_{x\leq y<z} p_{x,y}\mu_{y,z}^s m\right) & z\in D_{J,s}^0 \\
		\sum\limits_{x\in D_{J,s}} T_x \otimes (-v_s-v_s^{-1}) p_{x,z}m & z\in D_{J,s}^-
	\end{cases}
\end{align*}
On right hand side we find:
\begin{align*}
	C_s c(z|m) &= \sum_{x\in D_{J,s}} T_s T_x\otimes p_{x,z}m + T_x\otimes (-v_s) p_{x,z}m \\
	&= \sum_{x\in D_{J,s}^+} T_{sx} \otimes p_{x,z}m + T_x\otimes (-v_s) p_{x,z}m \\
	&\quad + \sum_{x\in D_{J,s}^0} T_x\otimes T_{s^x}p_{x,z}m + T_x\otimes (-v_s) p_{x,z}m \\
	&\quad + \sum_{x\in D_{J,s}^-} (T_{sx}+(v_s-v_s^{-1})T_x)\otimes p_{x,z}m + T_x\otimes (-v_s) p_{x,z}m \\
	&= \sum_{x\in D_{J,s}^-} T_{x} \otimes p_{sx,z}m + \sum_{x\in D_{J,s}^+} T_x\otimes(-v_s) p_{x,z}m \\
	&\quad + \sum_{x\in D_{J,s}^0} T_x\otimes (T_{s^x}-v_s)p_{x,z}m \\
	&\quad + \sum_{x\in D_{J,s}^+} T_x\otimes p_{sx,z}m + \sum_{x\in D_{J,s}^-} T_x\otimes(-v_s^{-1})p_{x,z}m \\
	&= \sum_{x\in D_{J,s}^+} T_x\otimes (p_{sx,z}-v_s p_{x,z})m \\
	&\quad + \sum_{x\in D_{J,s}^0} T_x\otimes C_{s^x}p_{x,z}m \\
	&\quad + \sum_{x\in D_{J,s}^-} T_{x} \otimes (p_{sx,z}-v_s^{-1}p_{x,z})m
\end{align*}
Comparing the $T_x\otimes M$ components we find an equation of elements of $\IZ[\Gamma]\Omega_J$ that needs to be satisfied. More specifically it is the equation in \ref{lemma:formula_p_mu} below.

\medbreak
Similarly the equations
\[\omega(e_s)\omega(x_s) = \omega(x_s) \quad\text{and}\quad \omega(x_s)\omega(e_s) = 0\]
translate into equations of elements of $\IZ[\Gamma]\Omega_J$, the two equations in \ref{lemma:emu_mu} and \ref{lemma:mue_zero} below.

%
\begin{lemma}\label{main_theorem_formulas}
The elements $p_{x,z}$, $\mu_{x,z}^s$ of $\IZ[\Gamma]\Omega_J$ satisfy the following equations:
\begin{enumerate}[label=\alph*.),ref=part \alph*. of Lemma \thetheoremnumber]
\item \label{lemma:formula_p_mu}
For all $z\in D_J$ and all $x\in D_J$
\[\hspace{-0.5em}\left.\begin{array}{lr}
	x\in D_{J,s}^+ & p_{sx,z} - v_s p_{x,z} \\
	x\in D_{J,s}^0 & C_{s^x} p_{x,z} \\
	x\in D_{J,s}^- & p_{sx,z} -v_s^{-1} p_{x,z}	
\end{array}\right\rbrace = \left\lbrace\begin{array}{ll}
	p_{x,sz} + \sum\limits_{x\leq y<z} p_{x,y}\mu_{y,z}^s & z\in D_{J,s}^+ \\
	p_{x,z} C_{s^z} + \sum\limits_{x\leq y<z} p_{x,y}\mu_{y,z}^s & z\in D_{J,s}^0 \\
	-(v_s+v_s^{-1})p_{x,z} & z\in D_{J,s}^-
\end{array}\right.\]
\item \label{lemma:pe}
For all $z\in D_{J,s}^0$ and all $x\in D_J$:
\[p_{x,z} e_{s^z} = \begin{cases}
	-v_s p_{sx,z} e_{s^z} & x\in D_{J,s}^+ \\
	e_{s^x} p_{x,z} e_{s^z} & x\in D_{J,s}^0 \\
	-v_s^{-1} p_{sx,z} e_{s^z} & x\in D_{J,s}^-
	\end{cases}\]
\item \label{lemma:emu_mu}
For all $z\in D_J$ and all $x\in D_{J,s}^0$: $e_{s^x} \mu_{x,z}^s = \mu_{x,z}^s$
\item \label{lemma:mue_zero}
For all $z\in D_{J,s}^0$ and all $x\in D_J$: $\mu_{x,z}^s e_{s^z} = 0$
\end{enumerate}
\end{lemma}

The proof of this lemma can be found in the appendix of this paper. It is inspired by Lusztig's proof of the analogous equations in \cite{lusztig1990canonical}, but is significantly longer.
\end{proof}

\subsection{First applications}

\subsubsection{Recovering well-known examples of induced modules}

We will start off by proving that our result encompasses several classic $W$-graph existence results, including Howlett and Yin's \citep[Theorem 5.1]{howlett2003inducingI}.

\begin{example}
Starting with $J=\emptyset$ and the regular module $\Omega_\emptyset=\IZ$ we obtain the special case $\HYI_\emptyset^S(\IZ) =: KL^S$. As an $H$-module this is isomorphic to $\Ind_{H_\emptyset}^H(H_\emptyset) = H$ and the basis $\Set{z\vert 1 | z\in W}$ is identified with the Kazhdan-Lusztig basis $\Set{C_z | z\in W}$ via the canonicalisation map.

Thus we recover Kazhdan and Lusztig's result (c.f. \citep[1.3]{KL}) that the regular $H$-module is induced by a $W$-graph. The elements $\mu_{x,y}^s\in \IZ[\Gamma]\Omega_\emptyset = \IZ[\Gamma]$ equal the $\mu$-values defined in \cite{KL} and \cite{lusztig2003hecke} (in the case of unequal parameters) up to a sign. The elements $p_{x,y}\in\IZ[\Gamma]$ are related to the Kazhdan-Lusztig polynomials via 
\[p_{x,y} = (-1)^{l(x)+l(y)} v^{L(x)-L(y)} \overline{P_{x,y}}\]
\end{example}

\begin{example}
Starting with an arbitrary $J\subseteq S$ and and arbitrary one-dimensional $\Omega_J$-module\footnote{Remember that all one-dimensional $H_J$-modules are given by a unique $W_J$-graph so that there is absolutely no difference between $H_J$- and $\Omega_J$-modules in this case} $M$, one obtains a $W$-graph structure on the induced module $\Ind_{H_J}^H(M)$. 

This module is called $M^J$ by Doedhar in \cite{Deodhar1987parabolic}. The elements $p_{x,y}\in\IZ[\Gamma]\Omega_J$ act on $M$ by multiplication with polynomials which are related to Deodhar's and Couillens's (c.f. \cite{couillens1999generalisation}) parabolic Kazhdan-Lusztig polynomials $P_{x,y}^J$ in a similar way as the polynomials in the previous example are related to the absolute Kazhdan-Lusztig polynomials.
\end{example}

We now show that Howlett-Yin induction is appropriately named, i.e. that it really recovers the construction in \cite{howlett2003inducingI}.
\begin{proposition}\label{recovering_howlett_yin_result}
Assume $L(s)=1$ for all $s\in S$.

\medbreak
Let $J\subseteq S$ be arbitrary and $M$ a $k\Omega_J$-module with a $k$-basis $\mathfrak{C}\subseteq M$ w.r.t. which $e_s$ acts diagonally (i.e. a module given by a $W$-graph) for all $s\in J$. Let $c: \Res_H^\Omega(\HYI_J^S(M)) \to \Ind_{H_J}^H(\Res_{H_J}^{\Omega_J}(M))$ be the canonicalisation isomorphism.

Then $c(x\vert\mathfrak{x})=C_{x,\mathfrak{x}}$ for all $x\in D_J, \mathfrak{x}\in\mathfrak{C}$ where $C_{x,\mathfrak{x}}$ denotes the canonical basis defined by Howlett and Yin in \citep[Theorem 5.1]{howlett2003inducingI}. The $W$-graphs given by the basis $\set{x\vert\mathfrak{x} | x\in D_J, \mathfrak{x}\in\mathfrak{C}}$ and by the basis $\set{C_{x,\mathfrak{x}} | x\in D_J, \mathfrak{x}\in\mathfrak{C}}$ are the same.
\end{proposition}
\begin{proof}
Our involution $\iota$ on $V:=H\otimes_{H_J} k[\Gamma]M$ is the same as Howlett and Yin's involution $\overline{\phantom{m}}$ defined in the introductory paragraph of \citep[Section 3]{howlett2003inducingI}.

Let $c: k[\Gamma]\HYI_J^S(M)\to V$ be the canonicalisation isomorphism and $\tilde{C}_{x,\mathfrak{x}}:=c(x|\mathfrak{x})$. Then $\Set{\tilde{C}_{x,\mathfrak{x}} | x\in D_J, \mathfrak{x}\in\mathfrak{C}}$ is a $k[\Gamma]$-basis of $V$ that satisfies $\overline{\tilde{C}_{x,\mathfrak{x}}} = \tilde{C}_{x,\mathfrak{x}}$ as well as $\tilde{C}_{x,\mathfrak{x}} \in T_x\otimes\mathfrak{x} + \sum_{w<x} T_w\otimes k[\Gamma_{>0}] M$. Howlett and Yin's theorem (as well as Theorem \ref{canonicalisation_from_shadows}) shows that there is a \emph{unique} basis with this property. Therefore $\tilde{C}_{x,\mathfrak{x}} = C_{x,\mathfrak{x}}$ as claimed.

That the $W$-graphs are identical follows from the fact that the $e_s$ act identical on both bases. In Theorem \ref{H_induction_as_Omega_module} we have chosen our definition such that
\begin{align*}
e_s \cdot x|\mathfrak{x} = x|\mathfrak{x} &\iff x\in D_{J,s}^- \vee (x\in D_{J,s}^0 \wedge e_{s^x}\mathfrak{x}=\mathfrak{x}) \\
&\iff (sx<x) \vee (sx>x \wedge sx\notin D_J \wedge s^x\in I(\mathfrak{x}))
\end{align*}
In \citep[Theorem 5.3]{howlett2003inducingI} the $W$-graph structure on $V$ is defined in such a way that (using Howlett and Yin's notation $\Lambda_s^-$)
\begin{align*}
e_s\cdot C_{x,\mathfrak{x}} = C_{x,\mathfrak{x}} &\iff (x,\mathfrak{x})\in\Lambda_s^- \\
&\iff (sx<x) \vee (sx>x \wedge sx\notin D_J \wedge s^x\in I(\mathfrak{x}))
\end{align*}
Because the canonicalisation map $c$ is $H$-linear, $H$ acts identical on both basis too. $x_s$ is a linear combination of $e_s$ and $T_s$ so that $x_s$ acts identical on both bases too, i.e. the edge weight matrices are also identical which proves that the two $W$-graphs are identical.
\end{proof}

\subsubsection{An algorithm to compute \texorpdfstring{$p$}{p} and \texorpdfstring{$\mu$}{µ}}

Note that \ref{lemma:formula_p_mu} and the recursive definition of the $\mu$ lead to the following recursive algorithm to compute $p_{x,z}$ and $\mu_{x,z}^s$ for all $x,z\in D_J$ and all $s\in S$.

The recursion is again along the (well-founded!) order $(x,z)\sqsubset(x',z') :\iff z<z' \vee (z=z' \wedge x>x')$ on $\Set{(x,z)\in D_J\times D_J | x\leq z}$.

\begin{algorithm}\label{algorithm:p_mu}
Input: $J\subseteq S$ and $x,z\in W$.

Output: $p_{x,z}\in \IZ[\Gamma]\Omega_J$ and $\mu_{x,z}^s \in \IZ[\Gamma]\Omega_J$ for all $s\in S$.

\begin{enumerate}[label=\arabic*.]
	\item If $x\not\leq z$, then $p_{x,z}=0$ and $\mu_{x,z}^s=0$.
	\item If $x=z$, then $p_{x,z}=1$ and $\mu_{x,z}^s=0$.
	\item If $x<z$, then choose any $t\in S$ with $tz<z$ and consider the following cases:
	\begin{enumerate}[label=\arabic{enumi}.\arabic*.]
		\item If $t\in D_J^+(x)$, then $p_{x,z} = -v_t p_{tx,z}$.
		\item If $t\in D_J^0(x)$, then $p_{x,z} = C_{t^x} p_{x,tz} - \sum_{y<tz} p_{x,y} \mu_{y,tz}^t$
		\item If $t\in D_J^-(x)$, then $p_{x,z} = p_{tx,tz} - v_t^{-1} p_{x,tz} - \sum_{y<tz} p_{x,y} \mu_{y,tz}^t$
	\end{enumerate}
	\item For all $s\in S$:
	\begin{enumerate}
		\item If $s\in D_J^+(x)$ or $s\in D_J^-(z)$, then $\mu_{x,z}^s = 0$.
		\item Otherwise compute $\alpha := -R-\sum_{x<y<z} p_{x,y} \mu_{y,z}^s$, where $R$ is defined as in \ref{lemma:def_mu}. Write $\alpha=\alpha_{-} + \alpha_0 + \alpha_{+}$ where $\alpha_{-}\in\IZ[\Gamma_{<0}]\Omega_J, \alpha_0\in\Omega_J, \alpha_{+} \in\IZ[\Gamma_{>0}]\Omega_J$. Then $\mu_{x,z}^s=\alpha_{-} + \alpha_0 + \overline{\alpha_{-}}$.
	\end{enumerate}
\end{enumerate}
\end{algorithm}

\subsubsection{More about the algebraic structure of \texorpdfstring{$\Omega$}{the W-graph algebra}}

The fact that $m\mapsto 1|m$ is an injective map $M\to\HYI_J^S(M)$ provides a simple proof to \citep[Conjecture 4.2.23]{hahn2013diss} from the author's thesis.

\begin{proposition}\label{corollary:parabolic_embedding}
Let $k$ be a commutative ring. Then the parabolic morphism $j: k\Omega_J\to k\Omega, e_s\mapsto e_s, x_{s,\gamma}\mapsto x_{s,\gamma}$ is injective.
\end{proposition}
\begin{proof}
Consider the Howlett-Yin induction $M:=\HYI_J^S(\Omega_J)$ of the regular $k\Omega_J$-module. It is a $k\Omega$-module so that $f: k\Omega \to M, a\mapsto a\cdot 1|1$ is a morphism of $k\Omega$-left-modules. For all $s\in J$ one finds $f(j(e_s)) = 1|e_s$ and $f(j(x_{s,\gamma})) = 1|x_{s,\gamma}$ so that $f(j(a))=1|a$ holds for all $a\in k\Omega_J$. In particular we find that $f\circ j$ is injective so that $j$ is injective.
\end{proof}
We will therefore suppress the embedding altogether and consider $k\Omega_J$ as a true subalgebra of $k\Omega$ from now on.

The Howlett-Yin induction also provides the correction to a small error in the proof of \citep[Corollary 4.2.19]{hahn2013diss} in the author's thesis.
\begin{proposition}
Let $k$ be a commutative ring. Define $E_J:=\prod_{s\in J} e_s \prod_{s\in S\setminus J} (1-e_s)\in k\Omega_J$. This element is non-zero in $k\Omega_J$.
\end{proposition}
The fallacious argument in my thesis considered the Kazhdan-Lusztig $W$-Graph $KL^S$ and assuming falsely that each $J\subseteq S$ occurs as a left descent set $D_L(w)$ for some $w\in W$ I concluded that $E_J$ must act non-trivially on this $W$-graph. This only works for finite Coxeter groups because a subset $J\subseteq S$ in fact occurs as a left descent set if and only if $W_J$ is finite (c.f. \citep[Lemma 3.6]{Deodhar1987parabolic}). In particular $S$ itself does not occur as a left descent set if $W$ is infinite. Nevertheless $S$ occurs in the $W$-graph of the sign representation and $E_S\in k\Omega$ is therefore non-zero. This is the idea of the following proof:
\begin{proof}
Consider the sign representation $M=k\cdot m_0$ of $k\Omega_J$, i.e. $e_s m_0 = m_0$, $x_s m_0 = 0$ for all $s\in J$. The element $1|m_0\in\HYI_J^S(M)$ satisfies:
\[e_s 1|m_0 = \begin{cases} 1|m_0 & s\in J \\ 0 & s\notin J\end{cases}\]
for all $s\in S$ so that $E_J\cdot 1|m_0 = 1|m_0$ and therefore $E_J\neq 0$.
\end{proof}

\section{Categorial properties of Howlett-Yin induction}\label{section:categorial_properties}

We now prove that Howlett-Yin induction is very well-behaved. It in particular it is a functor between module categories, given by tensoring with a certain bimodule, satisfies a transitivity property and a Mackey-type theorem.

\subsection{Howlett-Yin induction as a functor between module categories}

\begin{theorem}\label{functoriality_HYI}
Let $M, M_1, M_2$ be $k\Omega_J$-modules,and $\phi: M_1\to M_2$ a $k\Omega_J$-linear map.
\begin{enumerate}
	\item Using the notation from the Theorem \ref{H_induction_as_Omega_module}, the map
	\[\HYI_J^S(\phi): \HYI_J^S(M_1)\to \HYI_J^S(M_2), z|m_1 \mapsto z|\phi(m_1)\]
	is $\Omega$-linear. In particular $\HYI_J^S$ is a functor $k\Omega_J\textbf{-Mod}\to k\Omega\textbf{-Mod}$.
	
	\item $\HYI_J^S(\phi)$ commutes with the two canonicalisations, that is the diagram in Figure \ref{fig:H_vs_Omega_induction} commutes.

	\begin{figure}[ht]
	\centering
	\begin{tikzpicture}
		\node (M1) at (0,2) {$\Res_H^\Omega \HYI_J^S(M_1)$};
		\node (M2) at (5,2) {$\Res_H^\Omega \HYI_J^S(M_2)$};
		\node (V1) at (0,0) {$\Ind_{H_J}^H \Res_{H_J}^{\Omega_J} M_1$};
		\node (V2) at (5,0) {$\Ind_{H_J}^H \Res_{H_J}^{\Omega_J} M_2$};
		
		\path[->,font=\scriptsize]
			(M1) edge node[above]{$\id\otimes \HYI_J^S(\phi)$} (M2)
			(V1) edge node[below]{$\Ind_{H_J}^H \phi$} (V2)
			(M1) edge node[left]{$c_{M_1}$} (V1)
			(M2) edge node[left]{$c_{M_2}$} (V2);
	\end{tikzpicture}
	\caption{Functoriality of Howlett-Yin induction}%
	\label{fig:H_vs_Omega_induction}%
	\end{figure}
	
	In other words: The canonicalisation $c$ is a natural isomorphism $\Res_H^\Omega \circ \HYI_J^S \to \Ind_{H_J}^H\circ\Res_{H_J}^{\Omega_J}$.
\end{enumerate}
\end{theorem}
\begin{proof}
That $\HYI_J^S(\phi)$ is $\Omega$-linear is readily verified with the definition of the $\Omega$-action.

\medbreak
$\Ind_{H_J}^H(\phi)$ is certainly $H$-linear. Therefore $\phi(T_x \otimes M_1)=T_x\otimes\phi(M_1)\subseteq T_x\otimes M_2$ holds for all $x\in D_J$. It also commutes with $\iota$. By functoriality of canonicalisation, the diagram in Figure \ref{fig:H_vs_Omega_induction} commutes.
\end{proof}

\begin{remark}
The ordinary induction functor $\Ind_{\Omega_J}^\Omega$ is given by a tensoring with the $\Omega$-$\Omega_J$-bimodule $\Omega$. It is therefore natural to ask whether $\HYI_J^S$ can be described as a tensor functor and whether it satisfies a property similar to the Hom-tensor adjunction. The next proposition answers these questions in the affirmative.
\end{remark}
 
\begin{proposition}
The functor $\HYI_J^S$ is exact, commutes with direct sums and satisfies
\begin{enumerate}
	\item There is a sub-$\Omega$-$\Omega_J$-bimodule $\mathfrak{I}_J^S\leq\Omega$ such that $\overline{a}\otimes m\mapsto a\cdot 1|m$ is a natural isomorphism $\Omega/\mathfrak{I}_J^S \otimes_{\Omega_J} M \to \HYI_J^S(M)$.
	\item $\HYI_J^S(M)$ has the following universal mapping property in $k\Omega\textbf{-mod}$:
	\[\Hom(\HYI_J^S(M), X) \isomorphic \Set{f: M\to\Res_{\Omega_J}^\Omega(X) | \mathfrak{I}_J^S \cdot f(M) = 0}\]
	where the isomorphism is given by $F\mapsto (m\mapsto F(1|m))$.
\end{enumerate}
\end{proposition}
\begin{proof}
It is clear from the definition that $\HYI_J^S$ is exact and commutes with direct sums.

The Eilenberg-Watts-theorem (which characterizes cocontinuous and (right)exact functors between module categories, c.f. \cite{eilenberg1960functors}, \cite{watts1960functors}) implies $\HYI_J^S \isomorphic Q\otimes_{\Omega_J} -$ for some $\Omega$-$\Omega_J$-bimodule $Q$. In fact the proof is constructive. It shows that one can choose $Q$ to be $\HYI_J^S(\Omega_J)$ and the isomorphism as $\HYI_J^S(\Omega_J)\otimes M\to \HYI_J^S(M),z|a\otimes m\mapsto z|am$.

Furthermore $\HYI_J^S(\Omega_J)$ is generated by the element $1|1$: The $\Omega_J$-submodule generated by $1|1$ is $1|\Omega_J$ and in general $1|M$ generates $\HYI_J^S(M)$ as an $\Omega$-module. Therefore $\HYI_J^S(\Omega_J)$ is isomorphic to some quotient $\Omega/\mathfrak{I}_J^S$ via $a+\mathfrak{I}_J^S \mapsto a\cdot 1|1$.

\medbreak
The universal property follows from this presentation of the functor: Tensor the exact sequence $\mathfrak{I}_J^S \to \Omega \to \Omega/\mathfrak{I}_J^S \to 0$ with $M$. Right exactness of $-\otimes M$ implies that
\[\mathfrak{I}_J^S \otimes_{\Omega_J} M \to \Omega \otimes_{\Omega_J} M \to \HYI_J^S(M)\to 0\]
is exact. This provides a universal property of $\HYI_J^S(M)$ as the quotient of $\Omega \otimes M$ modulo the image of $\mathfrak{I}_J^S\otimes M \to \Omega\otimes M$. Combining this with Hom-tensor-adjunction $\Hom(\Omega\otimes M,X) \isomorphic \Hom(M,\Res_{\Omega_J}^\Omega(X))$ we obtain the result.
\end{proof}

\subsection{Transitivity}

\begin{remark}
Having a concept of \enquote{induction} directly leads to ask additional questions such as whether this is a transitive construction. Howlett and Yin did not address this question in their original papers. If I were to guess I'd say that the proliferation of indices and combinatorial formulas made such a proof infeasible. Here our more abstract approach to induction pays off by encapsulating all the work with recursive formulas involving $p$ and $\mu$.
\end{remark} 

\begin{lemma}\label{H_and_Omega_linearity}
Let $V_1,V_2$ be two $k\Omega$-modules and $f: V_1 \to V_2$ a $k$-linear map.

Then $f$ is $k\Omega$-linear if and only if the induced map $k[\Gamma]\otimes_k V_1 \to k[\Gamma]\otimes_k V_2$ is $H$-linear and $f(e_s m)=e_s f(m)$ holds for all $m\in V_1$.
\end{lemma}
\begin{proof}
Because $T_s = -v_s^{-1} e_s + v_s e_s + x_s$ the assumptions imply $f(x_s m) = x_s f(m)$ as elements of $k[\Gamma]\otimes V_2=\bigoplus_\gamma v^\gamma V_2$. Now by definition $x_s=\sum_{\gamma} x_{s,\gamma} v^\gamma$ so that $\sum_\gamma f(x_{s,\gamma} m) v^\gamma = \sum_\gamma x_{s,\gamma} f(m) v^\gamma$. Comparing coefficients gives $\Omega$-linearity. The reverse implication is clear.
\end{proof}

\begin{theorem}\label{transitivity}
Howlett-Yin-Induction is transitive. More precisely: If $J\subseteq K\subseteq S$, then
\[\tau_M:\HYI_K^S(\HYI_J^K(M)) \to \HYI_J^S(M), w|z|m \mapsto wz|m\]
is a natural $k\Omega$-module isomorphism.
\end{theorem}

\begin{proof}
\newlength{\lenarrow}
\settowidth{\lenarrow}{\scriptsize$\Ind_{H_K}^{H_S}(c_M)$}

Consider the diagram in Figure \ref{fig:transitivity_hyi}. Here $t: \Ind_{H_K}^{H_S} \circ \Ind_{H_J}^{H_K} \to \Ind_{H_J}^{H_S}$ is the natural isomorphism mapping $h_1\otimes(h_2\otimes m) \mapsto h_1 h_2 \otimes m$.

\begin{figure}[ht]
\centering
\begin{tikzpicture}
	\node (HYHYM) at (0,4) {$\Res_{H_S}^{\Omega_S} \HYI_K^S \HYI_J^K M$};
	\node (IndHYM) at (0,2) {$\Ind_{H_K}^{H_S} \Res_{H_K}^{\Omega_K} \HYI_J^K M$};
	\node (IndIndV) at (0,0) {$\Ind_{H_K}^{H_S} \Ind_{H_J}^{H_K} \Res_{H_J}^{\Omega_J} M $};
	\node (IndV) at (5,0) {$\Ind_{H_J}^{H_S} \Res_{H_J}^{\Omega_J} M$};
	\node (HYM) at (5,4) {$\Res_{H_S}^{\Omega_S} \HYI_J^S M$};
	
	\path[->,font=\scriptsize]
		(HYHYM) edge node[above]{$\id_{k[\Gamma]}\otimes\tau_M$} (HYM)
		(IndIndV) edge node[below]{$t_M$} (IndV)
		(HYM) edge node[right]{$c_M$} (IndV)
		(HYHYM) edge node[left]{$c_{\HYI_J^K(M)}$} (IndHYM)
		(IndHYM) edge node[left]{$Ind_{H_K}^{H_S}(c_M)$} (IndIndV);
\end{tikzpicture}
\caption{Transitivity of Howlett-Yin-Induction}%
\label{fig:transitivity_hyi}%
\end{figure}

We will show that this diagram commutes. Note that $t$ and $c$ are natural $H$-linear isomorphisms. In particular this expresses $\id_{k[\Gamma]}\otimes\tau_M$ as a composition of $H_S$-linear natural isomorphisms. We will easily verify that $\tau_M$ is in fact $\Omega_S$-linear so that $\tau_M$ really is a natural isomorphism between $\Omega_S$-modules.

\medbreak
To prove the diagram commutes we will show that the counter-clockwise composition of arrows from $\Res_{H_S}^{\Omega_S} \HYI_J^S M$ to $\Ind_{H_J}^{H_S} \Res_{H_J}^{\Omega_J} M$ equals the canonicalisation $c_M$.

First note that all maps in the diagram are in fact $\widehat{k[\Gamma]}$-linear: $\id_{k[\Gamma]}\otimes\tau_M$ is trivially $\widehat{k[\Gamma]}$-linear. $c_{\HYI_J^K(M)}$ is $\widehat{k[\Gamma]}$-linear because it is a canonicalisation. $\Ind_{H_K}^{H_S}(c_M)$ is $\widehat{k[\Gamma]}$-linear because $c_M$ is and $\Ind_{H_K}^{H_S}$ maps $\widehat{k[\Gamma]}$-linear maps to $\widehat{k[\Gamma]}$-linear maps. That $t_{\Res_{H_K}^{\Omega_K}(M)}$ is also $\widehat{k[\Gamma]}$-linear can readily be verified.

\medbreak
Next we prove that the counter-clockwise composition maps $xy|m$ into $T_{xy}\otimes m + \sum_{w\in D_J^S} T_w \otimes k[\Gamma]_{>0} M$ for all $(x,y)\in D_K^S\times D_J^K$:

\begin{align*}
	1\otimes xy|m &\xmapsto{\mathmakebox[\lenarrow]{\id\otimes\tau_M^{-1}}} 1\otimes x|y|m \\
	&\xhookrightarrow{\mathmakebox[\lenarrow]{c_{\HYI_J^K(M)}}} T_x\otimes y|m + \sum_{u\in D_K^S} T_u\otimes k[\Gamma]_{>0} \HYI_J^K(M) \\
	&\xhookrightarrow{\mathmakebox[\lenarrow]{\Ind_{H_K}^{H_S}(c_M)}} T_x\otimes T_y\otimes m + \sum_{(u,v)\in D_K^S \times D_J^K} T_u\otimes T_v\otimes k[\Gamma]_{>0} M \\
	&\xhookrightarrow{\mathmakebox[\lenarrow]{t_{\Res_{H_J}^{\Omega_J}(M)}}} T_{xy}\otimes m + \sum_{uv\in D_J^S} T_{uv}\otimes k[\Gamma]_{>0} M
\end{align*}
Theorem \ref{canonicalisation_from_shadows} shows that $c_M$ is the only $\widehat{k[\Gamma]}$-linear map that maps $1\otimes xy|m$ into $T_{xy}\otimes m + \sum_{w\in D_J^S} T_w\otimes k[\Gamma]_{>0} M$. This completes our proof that $\id_{k[\Gamma]}\otimes\tau_M$ is $H$-linear and natural in $M$.

\medbreak
Furthermore $\tau_M(e_s x|y|m) = e_s \tau_M(x|y|m)$ follows directly from lemma \ref{lemma:deodhar} and the definition of the $\Omega$-action so that $\tau_M$ is $\Omega$-linear by proposition \ref{H_and_Omega_linearity}.
\end{proof}

\subsection{The Mackey theorem for Howlett-Yin induction}

\begin{remark}
The next natural question is whether there exists a Mackey decomposition for the Howlett-Yin induction. Recall that the Mackey formula for group representations says
\[\Res_{W_K}^W \Ind_{W_J}^W(V) = \bigoplus_{d} \Ind_{W_K \cap {^d W_J}}^{W_K} \Res_{W_K\cap{^d W_J}}^{^d W_J}(^d V)\]
for all $k[W_J]$-modules $V$. Here the sum runs over a set of representatives $d$ for $W_K$-$W_J$-double cosets and $^d(-)$ denotes conjugation by $d$. The conjugated representation $^d V$ is isomorphic as a $k[W_K\cap{^d W_J}]$-module to $d\otimes V \subseteq \Ind_{W_J}^W(V)$.

One can show that parabolic subgroups of Coxeter groups are well-behaved in that $W_K\cap{^d W_J} = W_{K\cap{^d J}}$ if one chooses $d$ of minimal length in its double coset.

\medbreak
A similar formula also holds at the level of Hecke algebras:
\[\Res_{H_K}^H \Ind_{H_J}^H(V) = \bigoplus_{d} \Ind_{H_{K \cap {^d J}}}^{H_K} (^d V)\]
for all $H_J$-modules $V$. Here $d$ runs over the set of representatives of $W_K$-$W_J$-double cosets of shortest length and $^d V$ is the $H_{K\cap{^d J}}$-module $T_d\otimes V \subseteq \Ind_{H_J}^H(V)$. The reason for both of these formulas is that $W$ decomposes as disjoint union of double cosets and $H$ decomposes as a direct sum of $H_K$-$H_J$-bimodules accordingly.

\medbreak
Unfortunately there is no reason to expect that $\HYI_J^S(M)$ decomposes into a direct sum over double cosets because $\Omega$, unlike $k[W]$ and $H$, does not have such a direct sum decomposition (and in fact $\HYI_J^S(M)$ can be an indecomposable $\Omega$-module). Instead we will find a \emph{filtration} indexed by the double cosets whose layers play the role of the direct summands in the Mackey decomposition.
\end{remark}

\begin{lemma}
Let $J,K\subseteq S$. Define $D_{KJ} := D_K^{-1}\cap D_J$. Then
\begin{enumerate}
	\item $D_{KJ}$ is a system representatives of $W_K$-$W_J$-double cosets in $W$. More precisely $d\in D_{KJ}$ if and only if it is the unique element of minimal length within its double coset.
	\item $D_J^S = \coprod\limits_{d\in D_{KJ}} D_{K\cap{^d\!J}}^K\cdot d$.
	\item For $d\in D_{KJ}$ and $x\in D_{K\cap{^d\! J}}^K$:
	\[D_J^\ast(xd)\cap K = D_{K\cap{^d\!J}}^\ast(x)\cap K\]
	where $\ast\in\Set{+,0,-}$.
\end{enumerate}
\end{lemma}
\begin{proof}
See \citep[2.1.6--2.1.9]{geckpfeiffer}
\end{proof}

\begin{remark}
Fix some $d\in D_{KJ}$. For any $\Omega_J$-module $M$ one can define an $\Omega_{K\cap{^d\!J}}$-module $^d M$ by 
\[e_s\cdot {^d m} := {^d}(e_{s^d} m) \quad\text{and}\quad x_s\cdot {^d m} := {^d}(x_{s^d} m).\]

Similarly for any $H_J$-module $V$ one can define a $H_{K\cap{^d\!J}}$-module $^d V$ by 
\[T_s\cdot {^d v} := {^d}(T_{s^d} v).\]
Note that $^d V$ is isomorphic to the $H_{K\cap{^d\!J}}$-submodule $T_d\otimes V\subseteq \Ind_{H_J}^{H_S} V$.
\end{remark}

\begin{theorem}\label{mackey_hyi}
Let $J,K\subseteq S$. Furthermore let $M$ be a $k\Omega_J$-module and $V:=\Res_{H_J}^{\Omega_J} M$ its associated $kH_J$-module. For all $d\in D_{KJ}$ define the following $k$-submodules of $\HYI_J^S(M)$ and $\Ind_{H_J}^{H_S}(V)$ respectively:
\[F^{\leq d} \HYI_J^S(M) := \sum_{\substack{a\in D_{KJ},\; a\leq d \\ w\in D_{K\cap {^a\!\! J}}^K}} wa|M\]
\[F^{\leq d} \Ind_{H_J}^{H_S}(V) := \sum_{\substack{a\in D_{KJ},\; a\leq d \\ w\in D_{K\cap {^a\!\! J}}^K}}  T_{wa}\otimes V\]
The following hold for all $d\in D_{KJ}$:
\begin{enumerate}
	\item $F^{\leq d} \HYI_J^S(M)$ is a $\Omega_K$-submodule, $F^{\leq d} \Ind_{H_J}^{H_S}(V)$ is a $H_K$- and $\widehat{k[\Gamma]}$-submodule and the canonicalisation map $c_M$ identifies these with each other.
	\item The map $\Psi_M^d: \HYI_{K\cap {^d\! J}}^K(^d M) \to F^{\leq d}\HYI_J^S(M) / F^{<d}\HYI_J^S(M)$ which is defined by $w|{^d m} \mapsto wd|m$ for all $w\in D_{K\cap{^d\!J}}^K$ is a natural isomorphism of $\Omega_K$-modules.
\end{enumerate}
\end{theorem}

\begin{remark}
In \cite{howlett2004inducingII} Howlett and Yin also proved a Mackey-style theorem, which is a bit weaker than what is claimed here. Howlett and Yin only identify sub-$W_K$-graphs of the induced $W$-graph and prove that they are the same as the $W_K$-graphs for summands appearing in the Mackey formula.

Speaking in terms of modules this proves that $\HYI_{K\cap{^d J}}^K(^d M)$ appears as some subquotient of $\Res_{\Omega_K}^{\Omega_S}(\HYI_J^S(M))$. This is similar to describing a module by listing its composition factors. The new theorem states not only that these composition factors arise somewhere in the module but also identifies an explicit filtration in which they arise. (And also generalises the result to the multi-parameter case and non-free modules)
\end{remark}

\begin{proof}
That $F^{\leq d}\HYI_J^S(M)$ is a $\Omega_K$-submodule follows directly from the definition of the $\Omega$-action on $\HYI_J^S(M)$ and the observation $w'd' \leq wd \implies d'\leq d$ for all $w,w'\in W_K$, $d,d'\in D_{KJ}$. That $F^{\leq d}\Ind_{H_J}^{H_S}(V)$ is a $H_K$-submodule follows from the fact that it is equal to $\sum_{a\leq d} \operatorname{span}\Set{T_w | w\in W_K a W_J}\otimes V$.

\medbreak
To prove the second claim, observe that the given map is certainly a $k$-linear bijection. Next we will show that it makes the diagram in Figure \ref{fig:mackey_hyi} commute.

\begin{figure}[ht]
\centering
\begin{tikzpicture}
	\node (FM) at (0,2) {$F^{\leq d} \HYI_J^S(M) / F^{<d} \HYI_J^S(M)$};
	\node (IndM) at (0,0) {$F^{\leq d} \Ind_{H_J}^{H_S}(V) / F^{<d} \Ind_{H_J}^{H_S}(V)$};
	\node (IndM2) at (6,0) {$\Ind_{H_{K\cap{^d\!J}}}^{H_K}({^d} V)$};
	\node (HYIM) at (6,2) {$\HYI_{K\cap{^d\!J}}^K (^d M)$};
	
	\path[->,font=\scriptsize]
		(FM) edge node[left]{$\overline{c_M}$} (IndM)
		(IndM2) edge node[below]{$T_{xd}\otimes m \leftarrow T_x\otimes {^d m}$} (IndM)
		(HYIM) edge node[right]{$c_{^d M}$} (IndM2)
		(HYIM) edge node[above]{$\Psi_M^d$} (FM);
\end{tikzpicture}
\caption{Mackey isomorphism for Howlett-Yin-Induction}%
\label{fig:mackey_hyi}%
\end{figure}

This will again be done by utilising the uniqueness of the canonicalisation map. Observe that all maps involved in the diagram are $\widehat{k[\Gamma]}$-linear bijections. It is also readily verified that both the counter-clockwise and the clockwise compositions map $w|{^d m}$ into $T_{wd}\otimes m + \sum_{x\in D_{K\cap^{d\!J}}^K} T_{xd}\otimes k[\Gamma]_{>0} M$ so that the diagram indeed commutes by uniqueness of the canonicalisation map.

Because canonicalisation is $H_K$-linear, we conclude that $\Psi_M^d$ is also $H_K$-linear. It follows directly from the above lemma that $\Psi_M^d(e_s \cdot w|{^d m}) = e_s \cdot \Psi_M^d(w|{^d m})$ for $s\in K$. Proposition \ref{H_and_Omega_linearity} implies again that $\Psi_M^d$ is indeed $\Omega_K$-linear.
\end{proof}

\section{More applications}\label{section:applications}

\subsection{An improved algorithm to compute \texorpdfstring{$p$}{p} and \texorpdfstring{$\mu$}{mu}}

Denote with $p_{x,z}^J$ and $\mu_{x,z}^{s,J}$ the elements in $\IZ[\Gamma]\Omega_J$ from Definition \ref{lemmadef:iota_p} and \ref{lemma:def_mu} respectively to make the dependence from $J\subseteq S$ explicit. Note that these elements do not depend on $S$ in the sense that for any parabolic subgroup $W_J\subseteq W_K\subseteq W$ with $x,z\in W_K$ the elements computed w.r.t. the inclusion $J\subseteq S$ are the same as when computed w.r.t. the inclusion $J\subseteq K$.

\begin{proposition}
Suppose $J\subseteq K\subseteq S$. Let $u,x\in D_K^S$ and $v,y\in D_J^K$. Furthermore let $\omega_J^K : \Omega_K \to \Omega_J^{D_J^K \times D_J^K}$ be the matrix representation induced by the action of $\Omega_K$ on $\HYI_J^K(\Omega_J) = \bigoplus_{v\in D_J^K} v|\Omega_J$.
\begin{enumerate}
	\item If $u\not\leq x$, then $\mu_{uv,xy}^{s,J}$ and $p_{uv,xy}^J = 0$.
	\item If $u=x$, then
	\[\mu_{uv,xy}^{s,J} = \begin{cases} \mu_{v,y}^{s^x,J} & s\in D_K^0(x) \\ 0 & \text{otherwise} \end{cases} \quad\text{and}\quad p_{uv,xy}^J = p_{v,y}^J\]
	\item If $u<x$, then
	\[\mu_{uv,xy}^{s,J} = \omega_J^K(\mu_{u,x}^{s,K})_{v,y}\quad\text{and}\quad p_{uv,xy}^J = \sum_{\substack{t\in D_J^K \\ v\leq t}} p_{v,t}^J \cdot \omega_J^K(p_{u,x}^K)_{ty}\]
\end{enumerate}
\end{proposition}
\begin{proof}
Considering that the diagram in Figure \ref{fig:transitivity_hyi} is commutative, one finds that $x|y|m \in \HYI_K^S(\HYI_J^K(M))$ is mapped both to
\[\sum_{u\leq x} \sum_{v\leq t} T_u T_v \otimes p_{v,t}^J \omega_J^K(p_{u,x}^K)_{ty} \cdot m
\quad\text{and to}\quad
\sum_{u,v} T_{uv}\otimes p_{uv,xy}^J\cdot m.\]
Setting $M=\Omega_J$ and $m=1$, we obtain the equations for $p_{uv,xy}^J$.

\medbreak
We consider the identification $\HYI_K^S \HYI_J^K(M) \isomorphic \HYI_J^S(M)$ from Theorem \ref{transitivity} and the action of $x_s$ on both modules:
\begin{align*}
	x_s \cdot xy|m &= \begin{cases}
		sxy|m + \sum_{uv} uv|\mu_{uv,xy}^{s,J} m & s\in D_J^+(xy) \\
		xy|x_{s^{xy}} m + \sum_{uv} uv|\mu_{uv,xy}^{s,J} m & s\in D_J^0(xy) \\
		0 & s\in D_J^-(xy)
	\end{cases} \\
	x_s \cdot x|y|m &= \begin{cases}
		sx|y|m + \sum_{u<x} u|\mu_{u,x}^{s,K}\cdot y|m & s\in D_K^+(x) \\
		x|x_{s^{x}}\cdot y|m + \sum_{u<x} u|\mu_{u,x}^{s,K}\cdot y|m & s\in D_K^0(x) \\
		0 & s\in D_K^-(x)
	\end{cases} \\
	&=  \begin{cases}
		sx|y|m + \sum_{u<x} u|\mu_{u,x}^{s,K}\cdot y|m & s\in D_K^+(x) \\
		x|s^{x}y|m + \sum_{v<y} x|v|\mu_{v,y}^{s^x,J} m + \sum_{u<x} u|\mu_{u,x}^{s,K}\cdot y|m & s\in D_K^0(x), s^x \in D_J^+(y) \\
		x|y|x_{s^{xy}} m + \sum_{v<y} x|v|\mu_{v,y}^{s^x,J} m + \sum_{u<x} u|\mu_{u,x}^{s,K}\cdot y|m & s\in D_K^0(x), s^x\in D_J^0(y) \\
		0 & s\in D_K^0(x), s^x\in D_J^-(y) \\
		0 & s\in D_K^-(x)
	\end{cases}
\end{align*}
Therefore
\[\sum_{uv} u|v|\mu_{uv,xy}^{s,J} m = \begin{cases}
	\sum_{u<x} u|\mu_{u,x}^{s,K}\cdot y|m & s\in D_K^+(x) \\
	\sum_{v<y} x|v|\mu_{v,y}^{s^x,J} m + \sum_{u<x} u|\mu_{u,x}^{s,K}\cdot y|m & s\in D_K^0(x), s^x \in D_J^+(y)\cup D_J^0(y) \\
	0 & \text{otherwise}
	\end{cases}\]
Comparing the component $u|\HYI_J^K(M)$, we find
\[\sum_v v|\mu_{uv,xy}^{s,J} m = \begin{cases}
	\mu_{u,x}^{s,K}\cdot y|m & u<x \\
	\sum_{v<y} v|\mu_{v,y}^{s^x,J} m & u=x, s\in D_K^0(x) \\
	0 & \text{otherwise}
	\end{cases}\]
Now set $M:=\Omega_J$ and $m:=1$.
\end{proof}

This suggests the following algorithm for computing $p_{w,z}^J$ and $\mu_{w,z}^{s,J}$ for all $w,z\in D_J^S$ and all $s\in S$:
\begin{algorithm}\label{algorithm:mu_inductive}
Input: $J\subseteq S$.

Output: $p_{w,x}, \mu_{w,x}^s\in\IZ[\Gamma]\Omega_J$ for all $w,x\in D_J$ and all $s\in S$.

\begin{enumerate}[label=\arabic*.]
	\item Choose a flag $J=K_0 \subsetneq K_1 \subsetneq \ldots \subsetneq K_n=S$.
	\item For all $i=0,\ldots,n-1$, all $u,x\in D_{K_i}^{K_{i+1}}$, and all $s\in K_{i+1}$ compute $p_{u,x}^{K_i}$ and $\mu_{u,x}^{s,K_i}\in\Omega_{K_i}$ with Algorithm \ref{algorithm:p_mu}.
	\item For $i=0,\ldots,n-1$ compute $p_{w,x}^J$ and $\mu_{w,z}^{s,J}$ for all $w,z\in D_J^{K_{i+1}}$ and all $s\in K_{i+1}$ as follows:
	\begin{enumerate}[label=\arabic{enumi}.\arabic*.]
		\item Write $w=uv$, $z=xy$ with $u,x\in D_{K_i}^{K_{i+1}}$ and $v,y\in D_J^{K_i}$.
		\item If $u\not\leq x$, then $\mu_{w,z}^{s,J} = 0$ and $p_{w,z}^J=0$.
		\item If $u=x$, then $\mu_{w,z}^{s,J} = \begin{cases} \mu_{v,y}^{s^x,J} & s\in D_{K_i}^0(x) \\ 0 & \text{otherwise} \end{cases}$ and $p_{w,z}^J = p_{v,y}^{K_i}$.
		\item If $u<x$, then compute $\omega_J^{K_i}(\mu_{v,y}^{s,K_i})$ and $\omega_J^{K_i}(p_{u,x}^{K_i})$. Assemble the $p_{v',y'}^J$ with $v',y'\in D_J^{K_i}$ into the matrix $P\in (\IZ[\Gamma_{\geq 0}] \Omega_J)^{D_J^{K_i} \times D_J^{K_i}}$.
		
		Then $\mu_{w,z}^{s,J} = \omega_J^{K_i}(\mu_{u,x}^{s,K_i})_{v,y}$ and $p_{w,z}^J = \left(P \cdot \omega_J^{K_i}(p_{u,x}^{K_i})\right)_{v,y}$.
	\end{enumerate}
\end{enumerate}
\end{algorithm}

\begin{remark}
Note that the action of $\Omega_{K_i}$ on $\HYI_J^{K_i}(\Omega_J)$ only involves values of $\mu_{w',z'}^{s',J}$ where $w',z'\in D_J^{K_i}$ and $s'\in K_i$ which are already known by the previous iteration of the loop.

\medbreak
The big advantage of this algorithm over a direct computation of all $\mu_{w,z}^{s,J}$ with algorithm \ref{algorithm:p_mu} is that the expensive recursion over $D_J^S$ is replaced by $n$ collectively cheaper recursions over $D_{K_0}^{K_1}, D_{K_1}^{K_2}, \ldots, D_{K_{n-1}}^{K_n}$ so that fewer polynomials $p_{w,z}$ need to be computed and the computed elements are less complex (measured for example by the maximal length of occurring words in the generators $e_s, x_s$ of $\Omega$) and therefore need less memory.

Additionally the $n$ calls to algorithm \ref{algorithm:p_mu} in step 2 are independent of each other and can be executed in parallel which can lead to a sizeable speed-up.
\end{remark}

\begin{remark}
The Mackey-isomorphism $\Psi^d$ from Theorem \ref{mackey_hyi} translates into the equation
\[\mu_{yd,wd}^{s,J} = \kappa_d(\mu_{y,w}^{s,K\cap{^d\!J}})\]
for all $d\in D_{KJ}$, all $y,w\in D_{K\cap{^d\!J}}^K$, and all $s\in K$ where $\kappa_d: \Omega_{K\cap{^d\!J}} \to \Omega_{K^d\cap J}$ is the isomorphism $e_s\mapsto e_{s^d}, x_s\mapsto x_{s^d}$.

Provided one knows all $\mu_{u,v}^{s,T}$ for all $T\subseteq K$, all $u,v\in D_T^K$, and all $s\in K$, one can use that knowledge to partially calculate $\mu_{y,w}^{s,J}$. This in turn might be used to give the recursion from algorithm \ref{algorithm:p_mu} a head start and reduce the necessary recursion depth.
\end{remark}

\subsection{Induction of left cells}

Transitivity of Howlett-Yin induction also enables us to effortlessly reprove a result of Geck regarding the induction of cells (see \cite{geck2003induction}).

\begin{remark}
Recall the definition of (left) Kazhdan-Lusztig cells: Define a preorder $\preceq_\mathcal{L}$ on $W$ by defining $\Set{y\in W | y\preceq_\mathcal{L} z}$ to be the smallest subset $\mathfrak{C}\subseteq W$ such that the subspace $\operatorname{span}_{\IZ[\Gamma]} \Set{C_y | y\in\mathfrak{C}}$ is a $H$-submodule of $H$. The preorder then defines a equivalence relation $\sim_\mathcal{L}$ as usual by $x\sim_\mathcal{L} y :\!\iff x\preceq_\mathcal{L} y\preceq_\mathcal{L} x$. The equivalence classes of this relation are called \emph{left cells}.
\end{remark}

\begin{proposition}
\begin{enumerate}
	\item $\mathfrak{C}\subseteq W$ is $\preceq_{\mathcal{L}}$-downward closed if and only if $\operatorname{span}_\IZ\Set{x\vert 1 | x\in\mathfrak{C}}$ is a $\Omega$-submodule of $\HYI_\emptyset^S(\Omega_\emptyset)$.
	\item If $\mathfrak{C}\subseteq W_J$ is a union of left cells, then $D_J^S\cdot\mathfrak{C}\subseteq W$ is also a union of left cells.
\end{enumerate}
\end{proposition}
\begin{proof}
Consider the Kazhdan-Lusztig-$W$-graph $\HYI_\emptyset^S(\Omega_\emptyset)$. Then $\Set{x\vert 1 | x\in W}$ constitute a $\IZ$-basis of this module which (under the canonicalisation map) corresponds to the basis $\Set{C_x | x\in W}$. That $\mathfrak{C}$ is a $\preceq_\mathcal{L}$-downward closed means that $\operatorname{span}_{\IZ[\Gamma]}\Set{C_x | x\in\mathfrak{C}}$ is a $H$-submodule of $H$. Because $e_s \cdot x|1 \in \Set{0,x\vert 1}$ for all $x$ and $s$, every subset of the form $\operatorname{span}_\IZ\Set{x\vert 1 | x\in\mathfrak{C}}$ is closed under multiplication with $e_s$. Since $e_s$ and $T_s$ together generate $\Omega$ this proves the first statement.

Now let $\mathfrak{C}\subseteq W_J$ be $\preceq_{\mathcal{L}}$-downward closed and $M:=\operatorname{span}_\IZ\Set{x\lvert 1 | x\in\mathfrak{C}}$ be the corresponding submodule of $\HYI_\emptyset^J(\Omega_\emptyset)$. Then $\HYI_J^S(M)=\operatorname{span}_\IZ\Set{w\vert x\vert 1 | w\in D_J^S, x\in\mathfrak{C}}$ is a submodule of $\HYI_J^S \HYI_\emptyset^J(\Omega_\emptyset) \isomorphic \HYI_\emptyset^S(\Omega_\emptyset)$. In other word $D_J^S \cdot \mathfrak{C}$ is a $\preceq_\mathcal{L}$-downward closed set of $W$. Because every union of cells can be written as a set difference $\mathfrak{C}_1 \setminus \mathfrak{C}_2$ for some downward closed sets $\mathfrak{C}_2 \subseteq \mathfrak{C}_1\subseteq W$, this proves the second statement.
\end{proof}

\begin{remark}
Modifying Algorithm \ref{algorithm:mu_inductive} such that only $\mu$-values for elements in $D_J^S\cdot \mathfrak{C}$ are computed, we recover Geck's PyCox algorithm for the decomposition into left cells.
\end{remark}
%
%
%
%
%
%

\appendix
\section{Proof of Lemma \ref{main_theorem_formulas}}
We will prove the four equations in Lemma \ref{main_theorem_formulas} simultaneously with a double induction. We will induct over $l(z)$ and assume that all four equations hold for all pairs $(x',z')$ with $l(z')<l(z)$. For any fixed $z$ we will induct over $l(z)-l(x)$. Observe that all equations are trivially true if $l(x)>l(z)+1$ because all occurring $p$ and $\mu$ are zero. We will therefore assume that the equations also hold for all pairs $(x',z)$ with $l(x')>l(x)$.

\begin{proof}[Proof of \ref{lemma:formula_p_mu}]
We denote with $f_{xz}$ the difference between the right hand side and the left hand side of the equation. Then by the above considerations:
\[c(\omega(C_s)z|m) - C_s c(z|m) = \sum_x T_x\otimes f_{xz} m\]
We will show $f_{xz}\in\IZ[\Gamma_{>0}]\Omega_J$ and conclude $f_{xz}=0$ using lemma \ref{positive_symmetric_implies_zero}. Note that both $c(\omega(C_s)z|m)$ as well as $C_s c(z|m)$ are $\iota$-invariant elements of $\Ind_{H_J}^H \Res_{H_J}^{\Omega_J} M$ because $c$ is $\widehat{k[\Gamma]}$-linear and $\iota(C_s) = C_s$.

\bigbreak
Case 1: $z\in D_{J,s}^+$.

Case 1.1.+1.2: $x\in D_{J,s}^0 \cup x\in D_{J,s}^-$

In both cases $f_{xz}\in\IZ[\Gamma_{>0}]\Omega_J$ by definition of $\mu$.

Case 1.3.: $x\in D_{J,s}^+$.

Observe that $p_{sx,z} = 1 \iff sx=z \iff x=sz \iff p_{x,sz} =1$ so that $p_{sx,z} - p_{x,sz}$ is always an element of $\IZ[\Gamma_{>0}]\Omega_J$. Thus
\begin{align*}
f_{xz} &= \underbrace{p_{z,sx} - p_{sx,z}}_{\in\IZ[\Gamma_{>0}]\Omega_J} + \underbrace{v_s p_{x,z}}_{\in\IZ[\Gamma_{>0}]} + \sum_{y} p_{x,y} \mu_{y,z}^s \\
	&\equiv \sum_{y} p_{x,y} \mu_{y,z}^s \mod \IZ[\Gamma_{>0}]\Omega_J \\
	&= \sum_{y\in D_{J,s}^+} p_{x,y} \underbrace{\mu_{y,z}^s}_{=0} + \sum_{\substack{y\in D_{J,s}^0 \\ x\leq y<z}} p_{x,y} \mu_{y,z}^s + \sum_{\substack{y\in D_{J,s}^- \\ x\leq y<z}} p_{x,y} \mu_{y,z}^s \\
\intertext{Now since $x\in D_{J,s}^+$ we cannot have $x=y$ in the both the second and third sum so that we can use \ref{lemma:formula_p_mu} for $(x,y)$ and \ref{lemma:emu_mu} for $(y,z)$ in the second sum as well as \ref{lemma:formula_p_mu} for $(x,y)$ in the third sum so that we obtain}
f_{xz} &= \sum_{\substack{y\in D_{J,s}^0 \\ x\leq y<z}} (-v_s^{-1})(p_{x,y} C_{s^y} - p_{sx,y} + \sum_{x\leq y'<y} p_{x,y'}\mu_{y',y}^s) (e_{s^y}\mu_{y,z}^s) + \sum_{y\in D_{J,s}^-} \underbrace{(p_{sx,y}}_{\in\IZ[\Gamma_{\geq 0}]} \underbrace{(-v_s)) \mu_{y,z}^s}_{\in\IZ[\Gamma_{>0}]} \\
	&\equiv \sum_{\substack{y\in D_{J,s}^0 \\ x\leq y<z}} (-v_s^{-1})(p_{x,y} C_{s^y} - p_{sx,y} + \sum_{y'} p_{x,y'}\mu_{y',y}^s) e_{s^y} \mu_{y,z}^s \mod \IZ[\Gamma_{>0}]\Omega_J \\
\intertext{Now $y<z$ so that we can use \ref{lemma:mue_zero} for $(y',y)$ and obtain}
f_{xz} &= \sum_{\substack{y\in D_{J,s}^0 \\ x\leq y<z}} (-v_s^{-1})(p_{x,y} C_{s^y} - p_{sx,y} + \sum_{y'} p_{x,y'}\underbrace{\mu_{y',y}^s) e_{s^y}}_{=0} \mu_{y,z}^s \\
	&= \sum_{\substack{y\in D_{J,s}^0 \\ x\leq y<z}} (-v_s^{-1})(p_{x,y} C_{s^y} e_{s^y} - p_{sx,y} e_{s^y}) \mu_{y,z}^s \\
	&= \sum_{\substack{y\in D_{J,s}^0 \\ x\leq y<z}} (-v_s^{-1})(p_{x,y} (-v_s-v_s^{-1}) e_{s^y} - p_{sx,y} e_{s^y}) \mu_{y,z}^s \\
\intertext{Applying \ref{lemma:pe} to $(x,y)$ we find}
f_{xz} &= \sum_{y\in D_{J,s}^0} (-v_s^{-1})(-v_s p_{sx,y} e_{s^y} (-v_s-v_s^{-1}) - p_{sx,y} e_{s^y}) \mu_{y,z}^s \\
	&= \sum_{y\in D_{J,s}^0} (-v_s^{-1})(v_s^2 p_{sx,y} e_{s^y}) \mu_{y,z}^s \\
	&= \sum_{y\in D_{J,s}^0} \underbrace{-p_{sx,y}}_{\in\IZ[\Gamma_{\geq 0}]\Omega_J} e_{s^y} \underbrace{v_s \mu_{y,z}^s}_{\in\IZ[\Gamma_{>0}]\Omega_J} \\
	&\equiv 0 \mod\IZ[\Gamma_{>0}]\Omega_J
\end{align*}

\bigbreak
Case 2: $z\in D_{J,s}^0$.

Case 2.1.+2.2: $x\in D_{J,s}^0 \cup x\in D_{J,s}^-$

In both cases $f_{xz}\in\IZ[\Gamma_{>0}]\Omega_J$ by definition of $\mu$.

Case 2.3. $x\in D_{J,s}^+$.

In this case $sx\neq z$ so that $p_{sx,z}\in\IZ[\Gamma_{>0}]\Omega_J$. Therefore:
\begin{align*}
	f_{xz} &= p_{x,z} C_{s^z} + \sum_{y<z} p_{x,y} \mu_{y,z}^s  - p_{sx,z} + v_s p_{x,z} \\
	&\equiv p_{x,z} T_{s^z} + \sum_{y<z} p_{x,y} \mu_{y,z}^s \mod \IZ[\Gamma_{>0}]\Omega_J \\
\intertext{Because $x':=sx$ satisfies $x'>x$ and $x'\in D_{J,s}^-$ we can use the induction hypothesis for $(x'
,z)$ and $(x',y)$ so that}
f_{xz} &= p_{sx',z} T_{s^z} + \sum_{y<z} p_{sx',y} \mu_{y,z}^s \\
	&= (\underbrace{p_{x',z} C_{s^z}+ v_s^{-1} p_{x',z}}_{=p_{x',z} T_{s^z}^{-1}} + \sum_{y'} p_{x',y'} \mu_{y',z}^s) T_{s^z} + \sum_{y\in D_{J,s}^+} p_{x,y} \underbrace{\mu_{y,z}^s}_{=0} \\
	&\quad+ \sum_{y\in D_{J,s}^0} (\underbrace{p_{x',y} C_{s^y} + v_s^{-1} p_{x',y}}_{=p_{x',y} T_{s^y}^{-1}} + \sum_{y'} p_{x',y'} \mu_{y',y}^s)\mu_{y,z}^s \\
	&\quad+ \sum_{y\in D_{J,s}^-} (-(v_s+v_s^{-1})p_{x',y} + v_s^{-1} p_{x',y})\mu_{y,z}^s \\
	&= \underbrace{p_{x',z}}_{\in\IZ[\Gamma_{>0}]\Omega_J} + \sum_{x'\leq y'<z} p_{x',y'} \mu_{y',z}^s T_{s^z} \\
	&\quad+ \sum_{\substack{y\in D_{J,s}^0 \\ x'\leq y}} (p_{x',y} T_{s^y}^{-1} + \sum_{x'\leq y'<y} p_{x',y'} \mu_{y',y}^s)\mu_{y,z}^s \\
	&\quad+ \sum_{y\in D_{J,s}^-} -\underbrace{p_{x',y}}_{\in\IZ[\Gamma_{\geq 0}]\Omega_J} \cdot \underbrace{v_s\mu_{y,z}^s}_{\in\IZ[\Gamma_{>0}]\Omega_J} \\
	&\equiv \sum_{x'\leq y'<z} p_{x',y'} \mu_{y',z}^s (-v_s^{-1} e_{s^z} + v_s (1-e_{s^z}) + x_{s^z}) \\
	&\quad \sum_{\substack{y\in D_{J,s}^0 \\ x'\leq y}} p_{x',y} T_{s^y}^{-1}\mu_{y,z}^s + \sum_{\substack{y\in D_{J,s}^0, y'\in D_J \\ x'\leq y'<y<z}} p_{x',y'} \mu_{y',y}^s \mu_{y,z}^s \mod \IZ[\Gamma_{>0}]\Omega_J \\
\intertext{Because $x<x'$ we can use \ref{lemma:mue_zero} for $(y',z)$ in the first sum, \ref{lemma:emu_mu} for $(y,z)$ in the second and third sum as well as \ref{lemma:mue_zero} for $(y',y)$ in the third sum to obtain}
f_{xz} &= \sum_{x'\leq y'<z} \underbrace{p_{x',y'}}_{\in\IZ[\Gamma_{\geq 0}]\Omega_J} \underbrace{\mu_{y',z}^s v_s}_{\in\IZ[\Gamma_{>0}]\Omega_J} (1-e_{s^z})  \\
	&\quad + \sum_{\substack{y\in D_{J,s}^0 \\ x'\leq y}} p_{x',y} \underbrace{T_{s^y}^{-1} e_{s^y}}_{=-v_s e_{s^y}} \mu_{y,z}^s + \sum_{\substack{y\in D_{J,s}^0, y\in D_J \\ x'\leq y'<y<z}} p_{x',y'} \underbrace{\mu_{y',y}^s e_{s^y}}_{=0} \mu_{y,z}^s \\
	&\equiv \sum_{\substack{y\in D_{J,s}^0 \\ x'\leq y}} \underbrace{p_{x',y}}_{\in\IZ[\Gamma_{\geq 0}]} (-e_{s^y}) \underbrace{v_s \mu_{y,z}^s}_{\in\IZ[\Gamma_{>0}]\Omega_J} \mod \IZ[\Gamma_{>0}]\Omega_J \\
	&\equiv 0 \mod\IZ[\Gamma_{>0}]\Omega_J
\end{align*}
which is what we wanted to prove.

\bigbreak
Case 3: $z\in D_{J,s}^-$.

Case 3.1.: $x\in D_{J,s}^-$.

In this case $sx\in D_{J,s}^+$ so that $sx\neq z$ and thus $p_{sx,z}\in\IZ[\Gamma_{>0}]\Omega_J$. We infer
\[f_{xz} = p_{sx,z} - v_s^{-1} p_{x,z} + v_s p_{x,z} + v_s^{-1} p_{x,z} = p_{sx,z} + v_s p_{x,z} \in \IZ[\Gamma_{>0}]\Omega_J\]

Case 3.2.: $x\in D_{J,s}^0$.

In this case the equation is equivalent to $(1-e_{s^x})p_{x,z} = 0$ by Proposition \ref{eigenspace_Ts}. If $x\not\leq z$ then this is vacuously true because $p_{x,z}=0$. 

Since $z\in D_{J,s}^-$ we can write $z=sz'$ for some $z'\in D_{J,s}^+$ with $z'<z$. We can apply \ref{lemma:formula_p_mu} to $(x,z')$ and find:
\begin{align*}
	(1-e_{s^x}) p_{x,z} &= (1-e_{s^x}) p_{x,sz'} \\
	&= \underbrace{(1-e_{s^x})(C_{s^x}}_{=C_{s^x}} p_{x,z'} - \sum_{y} p_{x,y} \mu_{y,z'}^s) \\
	&= - \sum_{y\in D_{J,s}^+} (1-e_{s^x}) p_{x,y} \underbrace{\mu_{y,z'}^s}_{=0} \\
	&\quad- \sum_{\substack{y\in D_{J,s}^0 \\ x\leq y<z'}} (1-e_{s^x}) p_{x,y} \mu_{y,z'}^s
	- \sum_{\substack{y\in D_{J,s}^- \\ x\leq y<z'}} (1-e_{s^x}) p_{x,y} \mu_{y,z'}^s \\
\intertext{Because $z'<z$ we can apply \ref{lemma:emu_mu} to $(y,z')$ in the second sum and \ref{lemma:formula_p_mu} to $(x,y)$ (in the equivalent formulation $(1-e_{s^x}) p_{x,y}=0$) in the third sum so that we obtain}
(1-e_{s^x}) p_{x,z} &= - \sum_{\substack{y\in D_{J,s}^0 \\ x\leq y<z'}} (1-e_{s^x}) p_{x,y} e_{s^y} \mu_{y,z'}^s
\end{align*}
Using \ref{lemma:pe} for $(x,y)$ all the summands vanish.

\medbreak
Case 3.3.: $x\in D_{J,s}^+$.

In this case $x':=sx$ satisfies $x'>x$ and $x'\in D_{J,s}^-$ so that we find
\begin{align*}
	f_{xz} &= p_{sx,z} - v_s p_{x,z} + v_s p_{x,z} + v_s^{-1} p_{x,z} \\
	&= p_{x',z} + v_s^{-1} p_{sx',z} \\
\intertext{which -- using the induction hypothesis for $(x',z)$ -- equals}
	&= p_{x',z} + v_s^{-1} (-v_s p_{x',z}) \\
	&= 0 \qedhere
\end{align*}
\end{proof}

Since we have now proven \ref{lemma:formula_p_mu} we can use it to prove the other equations:
\begin{proof}[Proof of \ref{lemma:pe} and \ref{lemma:mue_zero}]
If $x\in D_{J,s}^{\pm}$ we can multiply the equation in \ref{lemma:formula_p_mu} with $e_{s^z}$ from the right:
\begin{align*}
	(p_{sx,z}  - v_s^{\pm 1} p_{x,z})e_{s^z} &= p_{x,z} C_{s^z} e_{s^z} + \mu_{x,z}^s e_{s^z} + \sum_{x<y<z} p_{x,y} \mu_{y,z}^s e_{s^z} \\
	&= p_{x,z} (-v_s-v_s^{-1}) e_{s^z} + \mu_{x,z}^s e_{s^z} + \sum_{x<y<z} p_{x,y} \mu_{y,z}^s e_{s^z} \\
\intertext{Now we can apply \ref{lemma:mue_zero} to $(y,z)$ in the sum and obtain:}
	&= p_{x,z} (-v_s-v_s^{-1}) e_{s^z} + \mu_{x,z}^s e_{s^z} \\
	\implies(p_{sx,z}  + v_s^{\mp 1} p_{x,z})e_{s^z}  &= \mu_{x,z}^s e_{s^z}
\end{align*}
If $x\in D_{J,s}^+$, then $\mu_{x,z}^s=0$ so that $p_{sx,z} = -v_s^{-1} p_{x,z}$ and $\mu_{x,z}^s e_{s^z} = 0$. If $x\in D_{J,s}^-$, then the left hand is contained in $\IZ[\Gamma_{>0}]\Omega_J$ because $sx,x\neq z$ while the right hand side is $\overline{\phantom{m}}$-invariant so that both sides equal to zero.

\medbreak
Now consider the case $x\in D_{J,s}^0$. Then the following holds by \ref{lemma:formula_p_mu} for $(x,z)$:
\begin{align*}
	(v_s+v_s^{-1})p_{x,z} e_{s^z} &= - p_{x,z} C_{s^z} e_{s^z} \\
	&= -C_{s^x} p_{x,z} e_{s^z} + \mu_{x,z} e_{s^z} + \sum_{x<y<z} p_{x,y} \mu_{y,z}^s e_{s^z}
\end{align*}
In the sum we can use \ref{lemma:mue_zero} for $(y,z)$ and find:
\begin{align*}
	(v_s+v_s^{-1})p_{x,z} e_{s^z} &= -C_{s^x} p_{x,z} e_{s^z} + \mu_{x,z} e_{s^z} \\
	\implies \mu_{x,z}^s e_{s^z} &= (C_{s^x} + v_s+v_s^{-1})p_{x,z}e_{s^z} \\
	&= ((v_s+v_s^{-1})(1-e_s) + x_s)p_{x,z}e_{s^z} \\
	&= ((v_s+v_s^{-1}) + x_s)(1-e_s)p_{x,z}e_{s^z}
\end{align*}
Now consider the exponents that occur at both sides of the equation: On the left hand side all exponents are $<L(s)$. On the right hand side this means $v_s (1-e_s)p_{x,z}e_{s^z} = 0$ because $p_{x,z}\in\IZ[\Gamma_{\geq 0}]\Omega_J$. Therefore $(1-e_s)p_{x,z}e_{s^z}=0$ and $\mu_{y,z}^s e_{s^z} = 0$.
\end{proof}

\begin{proof}[Proof of lemma \ref{lemma:emu_mu}]
Wlog we assume $x<z$ and $z\in D_{J,s}^+\cup D_{J,s}^0$ since otherwise $\mu_{x,z}^s=0$. Because $\overline{\mu_{x,z}^s}=\mu_{x,z}^s$ we only need to prove $(1-e_{s^x})\mu_{x,z}^s \equiv 0 \mod\IZ[\Gamma_{>0}]\Omega_J$.

Lemma \ref{lemma:formula_p_mu} for $(x,z)$ implies
\begin{align*}
0	&= (1-e_{s^x}) C_{s^x} p_{x,z} \\
	&= (1-e_{s^x}) p_{x,x} \mu_{x,z}^s + \left\lbrace\begin{array}{ll}
		(1-e_{s^x}) p_{x,sz}  & z\in D_{J,s}^+ \\
		(1-e_{s^x}) p_{x,z} C_{s^z} & z\in D_{J,s}^0
		\end{array}\right\rbrace + \sum_{x<y<z} (1-e_{s^x})p_{x,y} \mu_{y,z}^s \\
\intertext{Because $x\in D_{J,s}^0$, $x=sz$ cannot hold in the first case so that $p_{x,sz}\in\IZ[\Gamma_{>0}]\Omega_J$. In the second case we use $C_{s^z} = e_{s^z} C_{s^z}$ and the now proven \ref{lemma:pe}. We obtain}
0	&\equiv (1-e_{s^x}) \mu_{x,z}^s + \sum_{\substack{y\in D_{J,s}^0 \\ x<y<z}} (1-e_{s^x})p_{x,y} \mu_{y,z}^s
	+ \sum_{\substack{y\in D_{J,s}^- \\ x<y<z}} (1-e_{s^x})p_{x,y} \mu_{y,z}^s \mod \IZ[\Gamma_{>0}]\Omega_J
\intertext{Note that $(1-e_{s^x})p_{x,y}=0$ for $y\in D_{J,s}^-$ by \ref{lemma:formula_p_mu} applied to $(x,y)$ and $\mu_{y,z}^s = e_{s^y} \mu_{y,z}^s$ for $y\in D_{J,s}^0$ by \ref{lemma:emu_mu}. Using \ref{lemma:pe} we find}
0	&\equiv (1-e_{s^x}) \mu_{x,z}^s + \sum_{\substack{y\in D_{J,s}^0 \\ x<y<z}} \underbrace{(1-e_{s^x})p_{x,y} e_{s^y}}_{=0} \mu_{y,z}^s \\
	&= (1-e_{s^x}) \mu_{x,z}^s \qedhere
\end{align*}
\end{proof}

%
%

\bibliography{canon}
\end{document}